\documentclass[11pt,letterpaper,leqno]{amsart}
\usepackage{amsfonts}
\usepackage{pgfplots}
\pgfplotsset{compat=1.18}
\definecolor{graphscarlet}{RGB}{244,0,121}
\definecolor{graphazure}{RGB}{0,139,255}
\definecolor{graphforest}{RGB}{0,238,145}

\usepackage[letterpaper,textwidth=5.75in,textheight=8in,top=1.625in,centering]{geometry}
\usepackage{amssymb,mathtools}
\usepackage{microtype,booktabs,enumitem,needspace,placeins}

\makeatletter
\def\subsection{\@startsection{subsection}{2}%
  \z@{.5\linespacing\@plus.7\linespacing}{.75\baselineskip}%
  {\normalfont\bfseries}}
\makeatother
\usepackage[colorlinks=true,linkcolor=blue,citecolor=blue,urlcolor=black]{hyperref}

\setlist{itemsep=3pt,topsep=5pt}
\numberwithin{equation}{section}
\newtheorem{theorem}{Theorem}[section]
\newtheorem{proposition}[theorem]{Proposition}
\newtheorem{lemma}[theorem]{Lemma}
\newtheorem{corollary}[theorem]{Corollary}
\theoremstyle{definition}
\newtheorem{definition}[theorem]{Definition}

\newtheorem{remark}[theorem]{Remark}
\newtheorem*{thoughtexperiment}{Thought experiment}
\newenvironment{acknowledgements}{\section*{Acknowledgements}}{}
\newcommand{\doi}[1]{{\hypersetup{urlcolor=black}\href{https://doi.org/#1}{\nolinkurl{#1}}}}
\renewcommand{\MR}[1]{{\hypersetup{urlcolor=black}\href{https://mathscinet.ams.org/mathscinet/article?mr=#1}{MR#1}}}

\newcommand{\dd}{\,\mathrm d}
\newcommand{\proofpart}[2]{\par\addvspace{\medskipamount}\noindent\textit{Part #1. #2.}\par\nobreak\smallskip\noindent}

\hypersetup{pdftitle={Oscillation of Second-Order Delay Equations with Positive Feedback}}

\begin{document}
\title[Oscillation of Second-Order Delay Equations with Positive Feedback]{%
Oscillation of Second-Order Delay Equations with Positive Feedback}
\date{}

\begin{abstract}
We study oscillation of the second-order delay equation
\[
 x''(t)=x(\tau(t)),
\]
with a nondecreasing delayed argument $\tau(t)\le t$ satisfying $\tau(t)\overset{t\to\infty}{\to}\infty$.  Set 
\begin{equation*}%
a=\liminf_{t\to\infty}(t-\tau(t)),\qquad
b=\limsup_{t\to\infty}(t-\tau(t)).\end{equation*}
We describe a continuous strictly decreasing function
\begin{equation*}\beta:[0,2/e]\longrightarrow[2/e,\sqrt2],\qquad
\beta(0)=\sqrt2,\qquad\beta(2/e)=2/e, \end{equation*}
such that, if
the equation has an eventually positive decreasing solution, then
\begin{equation*}a\leq2/e,\qquad b\le\beta(a).\end{equation*}
The bound $b\le\beta(a)$ is sharp and
interpolates
between the classical \(\limsup\) threshold and the \(\liminf\)  autonomous
threshold.  A detailed spectral analysis gives a simple interpretation of the result and links it directly to the oscillation theory of first-order delay equations.
\end{abstract}

\author{John Ioannis Stavroulakis}
\address{School of Mathematics, Georgia Institute of Technology, Atlanta, GA 30332}
\email{jstavroulakis3@gatech.edu}

\maketitle

\section{Introduction}

Consider the second-order nonautonomous delay equation
\begin{equation}  \label{eq:positive-general}
x''(t)=p(t)x(\tau(t)),\end{equation}
where the coefficient $p:\mathbb R\to\mathbb [0,\infty)$ is locally integrable, and the delayed argument $\tau:\mathbb R\to\mathbb R$ is measurable
and satisfies
\[
 \tau(t)\le t,\qquad \lim_{t\to\infty}\tau(t)=\infty.
\]
Throughout the paper, differential equations
and inequalities are understood in the Carath\'eodory sense.

Oscillation of \eqref{eq:positive-general} refers to nonexistence of
eventually positive nonincreasing solutions.
The study of oscillation for this equation goes back to Kamenskii and Myshkis~\cite{Kamenskii,Myshkis1972}. For \eqref{eq:positive-general}, we assume throughout that $\int^\infty p=\infty$.
Under this assumption, positive solutions either tend to infinity or eventually
decrease to zero, in which case $x'(t)\to0$ \cite{Kamenskii,Ladde}.  We can distinguish between \(\limsup\) and \(\liminf\) oscillation criteria.

\begin{proposition}[{\cite{Gustafson,Ladde}}]
\label{prop:BDS16}
Assume that $p(t)\geq0$, $t\in\mathbb R$, the delayed argument $\tau$
is nondecreasing, and that
\[
 \limsup_{t\to\infty}\int_{\tau(t)}^t(s-\tau(t))p(s)\dd s>1.
\]
Then all positive solutions of \eqref{eq:positive-general} are unbounded.
\end{proposition}
Related \(\limsup\) criteria appear in~\cite{LLP}. We also recall a classical comparison theorem. 

\begin{proposition}[{\cite{ABBD,BDK,LabovskiThesis}}]
\label{prop:BDS17}
Assume that $p(t)\geq0$, $t\in\mathbb R$, and $\tau(t)\le t$. Put
\[
 (L_0x)(t):=x''(t)-p(t)\mathbf 1_{\{\tau(t)\ge0\}}x(\tau(t)),
 \qquad t\ge0.
\]
Let $\mathcal W_0$ denote the Wronskian of the fundamental system of $L_0x=0$.
\begin{enumerate}
\item The following are equivalent:
\begin{enumerate}
\item there exists a positive nonincreasing solution $x$ of
\[
 L_0x\ge0\qquad\text{on }[0,\infty).
\]
\item there exists a positive nonincreasing solution $x$ of
\[
 L_0x=0\qquad\text{ on }[0,\infty).
\]
\item $\mathcal W_0(t)\ne0$ for every $t\ge0$.
\end{enumerate}

\item Fix $\omega>0$. The following are equivalent:
\begin{enumerate}
\item there exists a function $v$ such that
\[
 v(t)>0,\qquad v'(t)<0\quad(0\le t<\omega),\qquad
 L_0v\ge0,
\]
and
\[
 |v(\omega)|+|v'(\omega)|>0.
\]
\item the boundary-value problem
\[
 L_0z=0,\qquad z(0)=1,\qquad z(\omega)=0,
\]
has a solution satisfying
\[
 z(t)>0,\qquad z'(t)<0\qquad(0\le t<\omega).
\]
\item $\mathcal W_0(t)\ne0$ for every $0\le t\le\omega$.
\end{enumerate}
\end{enumerate}
\end{proposition}
Applying Proposition~\ref{prop:BDS17}, one obtains the following \(\liminf\) oscillation criteria.

\begin{proposition}[{\cite{BDK,Labovski1975,LabovskiThesis}}]
\label{prop:autonomous-liminf}
\leavevmode
\begin{enumerate}[label=\textup{(\arabic*)},leftmargin=1.45em,labelsep=.35em,itemsep=4pt,topsep=3pt]
\item Assume that $p(t)\ge0$, $t\in\mathbb R$, and that
\[
 \left(\operatorname*{ess\,sup}_{t\in\mathbb R}(t-\tau(t))\right)
 \sqrt{\operatorname*{ess\,sup}_{t\in\mathbb R}p(t)}
 \le\frac{2}{e}.
\]
Then \eqref{eq:positive-general} possesses a positive nonincreasing
solution on $[0,\infty)$.
\item For every $d>2/e$ there is a finite number $L(d)$ such that no
solution of
$
 y''(t)=y(t-d)
$
can remain positive and nonincreasing throughout an interval of length
$L(d)$ together with its preceding history of length $d$. Consequently, if
\begin{equation*}\liminf_{t\to\infty}(t-\tau(t))>2/e,\end{equation*}
 then
$x''(t)=x(\tau(t))$ has no positive nonincreasing solution.
\end{enumerate}
\end{proposition}

For the normalized equation
\begin{equation}\label{eq:positive-normalized}
x''(t)=x(\tau(t)),\end{equation}
write
\begin{equation*}%
a=\liminf_{t\to\infty}(t-\tau(t)),\qquad
b=\limsup_{t\to\infty}(t-\tau(t)).\end{equation*}
\Needspace{10\baselineskip}
Inspired by \cite{PSS2023}, we prove:

\begin{theorem}
\label{thm:positive-main} Assume that $\tau$ is nondecreasing, $\tau(t)\le t
$, and $\tau(t)\overset{t\to\infty}{\to}\infty$. There is a continuous strictly decreasing function
\begin{equation*}\beta:[0,2/e]\longrightarrow[2/e,\sqrt2],\qquad
\beta(0)=\sqrt2,\qquad\beta(2/e)=2/e, \end{equation*}
characterized in Lemmas~\ref{lem:positive-short}--\ref{lem:positive-boundary},
such that, if
\eqref{eq:positive-normalized} has an eventually positive decreasing solution, then
\begin{equation*}a\leq2/e,\qquad b\le\beta(a).\end{equation*}
The bound $b\le\beta(a)$ is sharp by Theorem~\ref{thm:positive-sharpness}.
\end{theorem}

\newcommand{\betacurvecoordinates}{%
(0,1.4142135623731)
(0.0063708248285583,1.4142134760815)
(0.012741649657117,1.4142128711961)
(0.019112474485675,1.4142112266495)
(0.025483299314233,1.4142080183656)
(0.031854124142791,1.4142027189133)
(0.03822494897135,1.4141947971571)
(0.044595773799908,1.4141837179003)
(0.050966598628466,1.4141689415229)
(0.057337423457024,1.4141499236122)
(0.063708248285583,1.4141261145861)
(0.070079073114141,1.4140969593068)
(0.076449897942699,1.4140618966867)
(0.082820722771257,1.4140203592832)
(0.089191547599816,1.4139717728831)
(0.095562372428374,1.4139155560744)
(0.10193319725693,1.4138511198068)
(0.10830402208549,1.4137778669367)
(0.11467484691405,1.4136951917587)
(0.12104567174261,1.413602479521)
(0.12741649657117,1.4134991059232)
(0.13378732139972,1.4133844365964)
(0.14015814622828,1.413257826564)
(0.14652897105684,1.4131186196805)
(0.1528997958854,1.4129661480495)
(0.15927062071396,1.4127997314162)
(0.16564144554251,1.4126186765356)
(0.17201227037107,1.4124222765125)
(0.17838309519963,1.4122098101129)
(0.18475392002819,1.4119805410435)
(0.19112474485675,1.4117337171983)
(0.19749556968531,1.4114685698689)
(0.20386639451386,1.4111843129174)
(0.21023721934242,1.4108801419074)
(0.21660804417098,1.4105552331914)
(0.22297886899954,1.4102087429519)
(0.2293496938281,1.4098398061912)
(0.23572051865666,1.4094475356675)
(0.24209134348521,1.4090310207737)
(0.24846216831377,1.408589326353)
(0.25483299314233,1.4081214914485)
(0.26120381797089,1.4076265279803)
(0.26757464279945,1.407103419346)
(0.27394546762801,1.406551118937)
(0.28031629245656,1.4059685485662)
(0.28668711728512,1.4053545967983)
(0.29305794211368,1.4047081171753)
(0.29942876694224,1.4040279263298)
(0.3057995917708,1.403312801976)
(0.31217041659935,1.4025614807681)
(0.31854124142791,1.4017726560152)
(0.32491206625647,1.4009449752417)
(0.33128289108503,1.4000770375773)
(0.33765371591359,1.3991673909642)
(0.34402454074215,1.398214529163)
(0.3503953655707,1.3972168885414)
(0.35676619039926,1.3961728446236)
(0.36313701522782,1.395080708379)
(0.36950784005638,1.3939387222258)
(0.37587866488494,1.3927450557214)
(0.3822494897135,1.3914978009084)
(0.38862031454205,1.3901949672827)
(0.39499113937061,1.3888344763447)
(0.40136196419917,1.3874141556904)
(0.40773278902773,1.3859317325936)
(0.41410361385629,1.3843848270245)
(0.42047443868485,1.3827709440427)
(0.4268452635134,1.3810874654947)
(0.43321608834196,1.3793316409361)
(0.43958691317052,1.3775005776875)
(0.44595773799908,1.3755912299221)
(0.45232856282764,1.3736003866651)
(0.45869938765619,1.3715246585702)
(0.46507021248475,1.3693604633176)
(0.47144103731331,1.3671040094524)
(0.47781186214187,1.3647512784556)
(0.48418268697043,1.3622980048049)
(0.49055351179899,1.3597396537437)
(0.49692433662754,1.3570713964261)
(0.5032951614561,1.3542880820516)
(0.50966598628466,1.351384206529)
(0.51603681111322,1.3483538771286)
(0.52240763594178,1.3451907724748)
(0.52877846077034,1.3418880971077)
(0.53514928559889,1.3384385296831)
(0.54152011042745,1.3348341636903)
(0.54789093525601,1.3310664393203)
(0.55426176008457,1.3271260648195)
(0.56063258491313,1.3230029252711)
(0.56700340974169,1.3186859762641)
(0.57337423457024,1.3141631192696)
(0.5797450593988,1.3094210547307)
(0.58611588422736,1.3044451077936)
(0.59248670905592,1.2992190201885)
(0.59885753388448,1.2937246998644)
(0.60522835871303,1.2879419174056)
(0.61159918354159,1.2818479347258)
(0.61797000837015,1.2754170466172)
(0.62434083319871,1.2686200087889)
(0.63071165802727,1.2614233160545)
(0.64487403131652,1.2437877735617)
(0.65712704604355,1.2263446153054)
(0.66772812023038,1.2091919996052)
(0.6768999670304,1.1924029174842)
(0.68483527360786,1.1760304638214)
(0.69170074921504,1.1601119829234)
(0.6976406275112,1.1446723530911)
(0.70277969670202,1.1297266044887)
(0.70722592115835,1.1152820152671)
(0.71107270959065,1.1013397956348)
(0.71440087743019,1.0878964439648)
(0.7172803446437,1.0749448401533)
(0.71977160465024,1.0624751273358)
(0.72192699519994,1.0504754223744)
(0.72379179791405,1.0389323873339)
(0.72540518958593,1.02783168781)
(0.72680106522854,1.0171583589923)
(0.72800875015936,1.0068970964136)
(0.72905361608266,0.99703248519911)
(0.72995761411211,0.98754917911521)
(0.7307397359318,0.97843203868382)
(0.73141641278386,0.96966623597742)
(0.73200186066497,0.9612373323634)
(0.7325083789838,0.95313133436114)
(0.7329466089537,0.94533473186784)
(0.73332575714908,0.93783452226018)
(0.7336537889222,0.93061822325977)
(0.73393759574363,0.92367387693791)
(0.73418313998198,0.9169900468102)
(0.73439558016465,0.91055580961865)
(0.73457937935097,0.90436074310657)
(0.73473839889457,0.89839491084792)
(0.73487597956483,0.89264884499146)
(0.73499501173158,0.88711352761265)
(0.73509799608754,0.88178037122769)
(0.73518709618429,0.87664119890956)
(0.7352641838853,0.87168822435159)
(0.73533087869107,0.86691403214607)
(0.73538858176248,0.86231155848183)
(0.73543850535716,0.85787407241277)
(0.73548169829723,0.85359515780637)
(0.73551906800357,0.84946869604769)
(0.73555139955936,0.84548884954629)
(0.73557937220357,0.84165004607171)
(0.73560357360079,0.83794696392601)
(0.73562451218717,0.83437451794808)
(0.73564262785198,0.83092784633435)
(0.73565830117897,0.82760229825244)
(0.73567186144197,0.8243934222193)
(0.73568359352238,0.8212969552105)
(0.73569374389416,0.81830881246506)
(0.73570252580184,0.81542507794966)
(0.73571012374055,0.81264199544399)
(0.73571669733192,0.80995596020959)
(0.73572238467755,0.80736351120659)
(0.73572730526031,0.80486132382019)
(0.73573156245448,0.80244620306437)
(0.73573524569752,0.80011507722905)
(0.73573843236901,0.79786499193837)
(0.73574118941629,0.79569310459153)
(0.7357435747609,0.79359667915685)
(0.73574563851548,0.7915730812949)
(0.73574742403651,0.78961977378326)
(0.73574896883517,0.78773431222227)
(0.73575030536546,0.78591434100003)
(0.73575146170592,0.78415758949657)
(0.73575246214959,0.78246186850897)
(0.73575332771434,0.78082506688174)
(0.73575407658442,0.77924514832519)
(0.7357547244925,0.77772014841002)
(0.73575528505018,0.77624817172196)
(0.73575577003398,0.77482738916802)
(0.7357561896327,0.77345603542089)
(0.73575655266149,0.77213240649154)
(0.73575686674706,0.77085485742111)
(0.7357571384879,0.76962180008739)
(0.73575737359288,0.76843170111033)
(0.73575757700121,0.76728307985784)
(0.73575775298621,0.76617450653619)
(0.73575790524507,0.76510460037265)
(0.73575803697653,0.7640720278685)
(0.73575815094807,0.76307550113553)
(0.73575824955407,0.76211377628761)
(0.73575833486609,0.76118565191438)
(0.73575840867643,0.76028996760104)
(0.73575847253572,0.75942560251908)
(0.73575852778555,0.75859147406396)
(0.73575857558666,0.75778653654585)
(0.73575861694326,0.75700977992175)
(0.7357586527242,0.75626022859974)
(0.73575868368119,0.7555369402589)
(0.73575871046458,0.75483900472399)
(0.73575873363706,0.75416554287019)
(0.73575875368545,0.7535157056053)
(0.73575877103093,0.75288867282148)
(0.73575878603791,0.7522836524353)
(0.73575879902166,0.75169987945753)
(0.73575881025495,0.75113661506574)
(0.73575881997378,0.75059314573176)
(0.73575882838233,0.75006878236102)
(0.73575883565724,0.74956285947795)
(0.73575884195136,0.74907473447174)
(0.73575884739691,0.74860378673942)
(0.7357588521083,0.74814941702424)
(0.7357588561845,0.74771104669153)
(0.73575885971115,0.74728811698198)
(0.73575886276234,0.74688008840095)
(0.73575886540217,0.74648644002353)
(0.7357588676861,0.74610666887464)
(0.73575886966212,0.74574028941746)
(0.73575887137173,0.74538683276996)
(0.73575887285085,0.74504584639972)
(0.73575887413056,0.74471689322264)
(0.73575887523773,0.7443995514842)
(0.73575887619564,0.7440934139531)
(0.73575887702441,0.74379808744432)
(0.73575887774144,0.74351319264894)
(0.7357588783618,0.74323836314275)
(0.73575887889853,0.74297324532362)
(0.73575887936289,0.74271749799986)
(0.73575887976465,0.74247079169389)
(0.73575888011225,0.74223280857166)
(0.73575888041298,0.7420032416597)
(0.73575888067317,0.74178179459966)
(0.73575888089828,0.74156818186829)
(0.73575888109304,0.7413621275872)
(0.73575888126154,0.74116336538226)
(0.73575888140733,0.74097163810875)
(0.73575888153346,0.74078669801688)
(0.73575888164258,0.74060830586876)
(0.735758881737,0.74043623059572)
(0.73575888181868,0.74027024952128)
(0.73575888188936,0.74011014771907)
(0.7357588819505,0.73995571703464)
(0.7357588820034,0.73980675820521)
(0.73575888204917,0.73966307686275)
(0.73575888208877,0.73952448804283)
(0.73575888212303,0.7393908103937)
(0.73575888215267,0.73926187183646)
(0.73575888217831,0.73913750302956)
(0.7357588822005,0.73901754477525)
(0.7357588822197,0.73890183845722)
(0.73575888223631,0.7387902362889)
(0.73575888225067,0.73868259185271)
(0.73575888226311,0.7385787640182)
(0.73575888227386,0.73847861969155)
(0.73575888228317,0.73838202816665)
(0.73575888229122,0.73828886306824)
(0.73575888229818,0.73819900267484)
(0.73575888230421,0.73811233279836)
(0.73575888230942,0.738028736094)
(0.73575888231394,0.73794810586222)
(0.73575888231784,0.73787033922758)
(0.73575888232122,0.73779533118492)
(0.73575888232414,0.73772298433786)
(0.73575888232666,0.73765320417843)
(0.73575888232885,0.73758590746212)
(0.73575888233074,0.73752099199458)
(0.73575888233238,0.7374583868265)
(0.73575888234288,0.73575888234288)
}

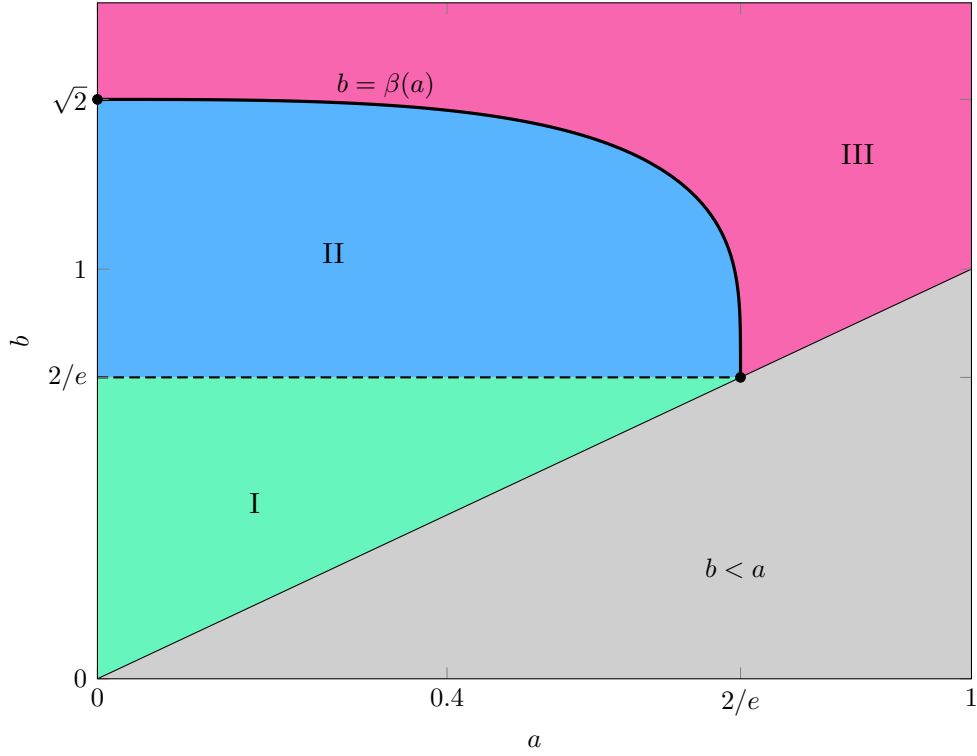
\begin{figure}[htbp]
\centering
\begin{tikzpicture}
\begin{axis}[
 width=0.90\textwidth,height=0.72\textwidth,
 xlabel={$a$},ylabel={$b$},xmin=0,xmax=1,ymin=0,ymax=1.65,
 xtick={0,0.4,0.735758882343,1},xticklabels={$0$,$0.4$,$2/e$,$1$},
 ytick={0,0.735758882343,1,1.414213562373},
 yticklabels={$0$,$2/e$,$1$,$\sqrt2$},
 tick label style={font=\small},label style={font=\small},
 axis on top,clip=true]
\addplot[draw=none,fill=gray!38] coordinates {(0,0)(1,0)(1,1)};
\addplot[draw=none,fill=graphforest!60] coordinates
 {(0,0)(0.735758882343,0.735758882343)(0,0.735758882343)};
\addplot[draw=none,fill=graphazure!65] coordinates {
(0,0.735758882343)
\betacurvecoordinates

};
\addplot[draw=none,fill=graphscarlet!60] coordinates {
\betacurvecoordinates
(1,1)
(1,1.65)
(0,1.65)
};
\addplot[black,very thick,no marks] coordinates {
\betacurvecoordinates

};
\addplot[black,thin,domain=0:1,samples=2] {x};
\addplot[black,thick,dash pattern=on 4pt off 2pt] coordinates {(0,0.735758882343)(0.735758882343,0.735758882343)};
\node[font=\small,anchor=south] at (axis cs:0.33,1.39) {$b=\beta(a)$};
\node at (axis cs:0.18,0.43) {I};
\node at (axis cs:0.27,1.04) {II};
\node at (axis cs:0.87,1.28) {III};
\node[font=\small] at (axis cs:0.73,0.27) {$b<a$};
\addplot[only marks,mark=*,mark size=1.8pt] coordinates
 {(0,1.414213562373)(0.735758882343,0.735758882343)};
\end{axis}
\end{tikzpicture}
\caption{The delay-data regions for positive feedback.
In region I, we have nonoscillation. In region II, either nonoscillation or
oscillation is possible. In region III, every bounded solution oscillates.}
\label{fig:beta}
\end{figure}

Lemmas~\ref{lem:positive-short}--\ref{lem:positive-boundary} determine $\beta$ explicitly in one parameter regime and by a scalar implicit
equation in the other. Thus the same criterion recovers the classical \(\limsup\)
bound of Proposition~\ref{prop:BDS16} at $a=0$ and the exact autonomous
threshold associated with Proposition~\ref{prop:autonomous-liminf} at $a=2/e$. The resulting
delay-data regions are shown in Figure~\ref{fig:beta}.

This paper is organized as follows. Section~\ref{subsec:positive-delay}
develops the extremal-profile method. It begins in Subsection~\ref{subsec:infinite-rank-feedback}
with a well-known fact about the spectral radius of positive operators, which gives an abstract interpretation of \(\liminf\)--\(\limsup\) dynamics. Next, we construct the extremal profile $\Phi_{a,b}$ (Subsection~\ref{subsec:positive-extremal-profile}), determine the critical curve $b=\beta(a)$ (Subsection~\ref{subsec:positive-critical-curve}), and prove that $b>\beta(a)$ is a sharp oscillation criterion (Subsection~\ref{subsec:positive-proof-sharpness}). The positive-operator and extremal-profile techniques used in this section are broad and can be applied to various similar systems.

Section~\ref{sec:spectral-positive} develops a complementary spectral approach that exploits the structure of our delay equation. Its starting point is that the positive decaying modes of the autonomous second-order equation can be encoded by a first-order delay equation with negative feedback, as explained in Remark~\ref{rem:spectral-first-order-reduction}. This reduction leads to precise spectral coordinates and asymptotics for positive decaying solutions (Subsection~\ref{subsec:spectral-description}). We establish spectral comparison principles between positive supersolutions and solutions of the corresponding equation (Subsection~\ref{subsec:spectral-comparison}) and use them to give a second proof of the oscillation bound in Subsection~\ref{subsec:spectral-critical-profile}. A robust return construction gives a second proof of sharpness in Subsection~\ref{subsec:spectral-excursions-sharpness}. 
\FloatBarrier
\section{The sharp \texorpdfstring{$\liminf$--$\limsup$}{liminf--limsup} boundary via positive operators}
\label{subsec:positive-delay}

\subsection{A lemma for positive operators}
\label{subsec:infinite-rank-feedback}

Let us begin with some abstract nonsense.

\begin{definition}
Let $X$ be a real ordered linear space with positive cone $X_+$. For
$e\in X_+\setminus\{0\}$, set
\[
 X_e^+:=\{x\in X_+:x\le ce\text{ for some }c\ge0\}.
\]
Let $T:X_e^+\to X_e^+$ be order-preserving and positively homogeneous,

\[
 x\le y\implies Tx\le Ty,
 \qquad T(cx)=cTx\quad(c\ge0).
\]
Define
\begin{equation*}
 \alpha_n(T;e):=\inf\{c\ge0:T^ne\le ce\},\qquad n\ge1,
\end{equation*}
and
\begin{equation}
 r(T;e):=\lim_{n\to\infty}\alpha_n(T;e)^{1/n}
       =\inf_{n\ge1}\alpha_n(T;e)^{1/n}.
\end{equation}
We call $r(T;e)$ the spectral radius of $T$ at $e$.
\end{definition}

The following lemma is well-known in positive-operator theory
\cite{AkianGaubertNussbaum2014,KantorovichAkilov1982,Krasnoselskii1964,KrasnoselskiiLifshitsSobolev1989,KrasnoselskiiEtAl1972,Schroder1980,Stetsenko1968}.
Standard formulations concern Banach spaces or finite-dimensional spaces.
We include a short proof for general ordered linear spaces.

\begin{lemma}

Let $X$ be a real ordered linear space, and let $A,B:X\to X$ be positive
linear operators. Suppose that $u\in X_+\setminus\{0\}$ satisfies
\begin{equation}\label{eq:infinite-rank-positive-super}
 u\ge Au+Bu.
\end{equation}
Assume that
\[
 \mathcal T x:=\sup_{N\ge0}\sum_{j=0}^{N}A^jBx
\]
exists in $X$ for every $x\in X_u^+$. Then
$\mathcal T:X_u^+\to X_u^+$ satisfies
\[
 \mathcal T u\le u,
\]
and
\begin{equation}\label{eq:infinite-rank-necessary}
 r(\mathcal T;u)\le1.
\end{equation}
\end{lemma}

\begin{proof}
Iterating \eqref{eq:infinite-rank-positive-super} and using positivity gives,
for every $N\ge0$,
\[
 u\ge A^{N+1}u+\sum_{j=0}^{N}A^jBu
   \ge\sum_{j=0}^{N}A^jBu.
\]
Taking the supremum yields $\mathcal T u\le u$. If $0\le x\le cu$,
then
\[
 0\le\mathcal T x\le c\mathcal T u\le cu.
\]
Thus $\mathcal T$ maps $X_u^+$ into itself. Finally, $\mathcal T u\le u$
and monotonicity give $\mathcal T^n u\le u$ for every $n\ge1$. Hence
$\alpha_n(\mathcal T;u)\le1$ and \eqref{eq:infinite-rank-necessary}.
\end{proof}

\begin{corollary}
\label{cor:rank-one-feedback}
Let $X$ be a real ordered linear space, let $A:X\to X$ be a positive
linear operator, let $\mathcal E:X\to\mathbb R$ be a positive linear
functional, and let $g\in X_+$. Assume that the following supremum exists
in $X$
\[
 \Phi:=\sup_{N\ge0}\sum_{j=0}^{N}A^jg.
\]
If $u\in X_+$ satisfies
\[
 u\ge Au+\mathcal E(u)g,
\]
then
\[
 u\ge \mathcal E(u)\Phi.
\]
Hence $\mathcal E(u)>0$ implies $\mathcal E(\Phi)\le1$.
Conversely, assume in addition that $\Phi=A\Phi+g$.
If $0<\mathcal E(\Phi)\le1$, then $z=\Phi/\mathcal E(\Phi)$ satisfies $\mathcal E(z)=1$ and $z\ge Az+g$.
\end{corollary}

\begin{proof}
Set $S_N=\sum_{j=0}^{N}A^jg$.  Iteration and positivity give
\[
 u\ge A^{N+1}u+\mathcal E(u)S_N\ge \mathcal E(u)S_N.
\]
Taking the supremum yields $u\ge \mathcal E(u)\Phi$. If $\mathcal E(u)>0$, applying $\mathcal E$ gives $\mathcal E(u)\ge \mathcal E(u)\mathcal E(\Phi)$, so $\mathcal E(\Phi)\le1$. For the converse, the relation $\Phi=A\Phi+g$ gives
\[
 z-Az-g
 =\frac{\Phi-A\Phi}{\mathcal E(\Phi)}-g
 =\left(\frac{1}{\mathcal E(\Phi)}-1\right)g\ge0.
\]
\end{proof}

\begin{remark}
\label{rem:rank-one-biological}
The rank-one case $Bf=\mathcal E(f)g$ has a standard interpretation in discrete-time
structured-population theory \cite{LiSchneider2002,RebarberTenhumbergTownley2012,SmithThieme2013,Thieme2023Turnover},
\cite[Section~12.3]{Thieme2024}. Here $A$ represents survival and transition among
individual states between reproductive events, while $B$ is the fertility or recruitment
operator. The positive functional $\mathcal E$ assigns to a population profile $f$ its total
production of recruits, and $g$ describes how the newly produced individuals are distributed among the possible states when they enter the population. Assume here that $X$ is a Banach lattice and $A$ is a bounded positive operator with $r(A)<1$. Then
\[
 \mathsf R_A=(I-A)^{-1}=\sum_{j=0}^{\infty}A^j
\]
is the lifetime survival or transition operator. Thus
$
 \Phi=\mathsf R_Ag
$
is the cumulative lifetime occupancy profile generated by one unit cohort of recruits.

The usual next-generation operator is \(B\mathsf R_A\), since it maps a cohort of recruits to the recruits produced over its lifetime. By contrast, \(\mathsf R_AB\) maps a population state to the cumulative lifetime occupancy generated by its recruits. With respect to positive supersolutions the two formulations are equivalent. Indeed, if
$u\ge Au+Bu$ and $v:=Bu$, then $u\ge\mathsf R_Av$ and hence
$v\ge B\mathsf R_Av$. Conversely, if $v\ge B\mathsf R_Av$ and
$u:=\mathsf R_Av$, then $u=Au+v\ge Au+Bu$. For both rank-one operators, the only possible
nonzero spectral value is
\[
 \mathcal E(\mathsf R_Ag)=\mathcal E(\Phi).
\]
Accordingly, $\mathcal E(\Phi)$ is the net reproduction number $R_0$, interpreted as the expected
number of recruits produced over the lifetime of one recruit. Finally, the inequality
$u\ge Au+\mathcal E(u)g$ says that $u$ is superinvariant under one survival-plus-recruitment
step. Thus $\mathcal E(\Phi)<1$ corresponds to under-replacement, while $\mathcal E(\Phi)=1$ corresponds
to exact replacement. \end{remark}

\subsection{The comparison equation}

Throughout this section, we consider the normalized equation
\eqref{eq:positive-normalized}, with $p\equiv1$:
\[
 x''(t)=x(\tau(t)).
\]
The delayed argument $\tau$ is nondecreasing, $\tau(t)\le t$, and
$\tau(t)\overset{t\to\infty}{\to}\infty$.
We write
\[
 a=\liminf_{t\to\infty}(t-\tau(t)),\qquad
 b=\limsup_{t\to\infty}(t-\tau(t)).
\]

\begin{thoughtexperiment}
Fix $0<a<2/e$ and $b\ge a$.  Choose a time at which the delay equals $b$, translate time so that this
point is $t=b$ and $\tau(b)=0$.
Then $\tau(t)\le t-a$ and, since $\tau$ is nondecreasing and $\tau(b)=0$,
\[
 \tau(t)\le0,\qquad t\le b.
\]
Hence a
positive decreasing solution satisfies
\begin{equation}\label{eq:positive-comparison}
 x''(t)\ge
 \begin{cases}
 x(t-a), & t\notin[a,b],\\
 x(0), & a\le t\le b
 \end{cases}.\end{equation}

\end{thoughtexperiment}

Much of this section is devoted to applying Corollary~\ref{cor:rank-one-feedback} to \eqref{eq:positive-comparison}. In the biological interpretation of Remark~\ref{rem:rank-one-biological}, nonoscillation is equivalent to a reproduction number at most $1$. 

It may be possible to obtain several of
the results of this section from known facts about  exponential dichotomy
and the representation of solutions to perturbed equations
\cite{BackesDragicevic2021,BackesDragicevicPituk2025,BackesEtAl2022,Coppel1978,Hale1977,HaleVerduynLunel1993,Schaeffer1975}.
However, as we are interested in a particular type of \emph{global}
solution, the application of such results may require significant modifications to
the framework and assumptions.

Instead, the methods of this
section are self-contained and expressed in the language of operator theory, cf.\
\cite{Azbelev1971,KantorovichAkilov1982,Krasnoselskii1964,KrasnoselskiiLifshitsSobolev1989,KrasnoselskiiEtAl1972,KreinRutman1948,Labovski1975,LabovskiThesis,Labovski1984,LabovskiAlves2018}. We now describe \eqref{eq:positive-comparison} in terms of operator equations.
Define the shift operator by
\[
 (\mathsf S_a f)(t):=f(t-a).
\]
For $\lambda>0$, we endow
\[
 X_\lambda=\left\{f\in C(\mathbb R):
 \sup_{t\in\mathbb R}e^{\lambda t}|f(t)|<\infty\right\}
\]
with the Banach norm
\[
 \|f\|_\lambda:=\sup_{t\in\mathbb R}e^{\lambda t}|f(t)|.
\]

\begin{lemma}[{\cite[p.~645, equations~(2.6)--(2.7)]{Atkinson1955}}]
\label{lem:terminal-green}
Let $g\in L^1_{\mathrm{loc}}(\mathbb R)$ and suppose that
$\int_t^\infty(1+u-t)|g(u)|\dd u<\infty$ for every $t\in\mathbb R$.
Then the positive operator
\[
 (\mathsf T g)(t)=\int_t^\infty(u-t)g(u)\dd u
\]
gives the unique solution of
\[
 y''=g\quad\text{a.e.},\qquad y(t),y'(t)\xrightarrow{t\to\infty}0.
\]
\end{lemma}
\begin{lemma}
\label{lem:positive-operator}
Let $0<a<2/e$, $b\ge a$, and $f\in L^1([a,b])$, extended by zero
outside $[a,b]$. Let $\lambda>0$ satisfy $e^{\lambda a}/\lambda^2<1$
(for example, $\lambda=2/a$). Define
\[
 G_a(v)=\sum_{j=0}^\infty\frac{(v+ja)_+^{2j+1}}{(2j+1)!}.
\]
\begin{enumerate}[label=\textup{(\roman*)},leftmargin=*,itemsep=\medskipamount]
\item\label{item:positive-resolvent}
For every $\Phi\in X_\lambda$,
\[
 \Phi(t)=\int_a^bG_a(u-t)f(u)\dd u
 \quad\Longleftrightarrow\quad
 \Phi=\mathsf T\mathsf S_a\Phi+\mathsf Tf.
\]
These equations have a unique solution in $X_\lambda$.
\item
For every $\Phi\in X_\lambda$,
\[
 \Phi(t)=\int_a^bG_a(u-t)[f(u)-\Phi(u-a)]\dd u
 \quad\Longleftrightarrow\quad
 \Phi=\mathsf T\mathbf1_{[a,b]^c}\mathsf S_a\Phi+\mathsf Tf.
\]
These equations have a unique solution in $X_\lambda$.
\end{enumerate}
\end{lemma}

\begin{proof}
For compactly supported integrable $f$, repeated integration and Fubini's
theorem give
\[
 (\mathsf T^{j+1}\mathsf S_{ja}f)(t)
 =\int_{\mathbb R}\frac{(u-t+ja)_+^{2j+1}}{(2j+1)!}f(u)\dd u,
 \qquad j\ge0.
\]
Indeed, $\mathsf T\mathsf S_a=\mathsf S_a\mathsf T$, so each iteration adds two
integrations and shifts the argument by $a$. The estimate
$\|\mathsf T\mathsf S_a f\|_\lambda\le (e^{\lambda a}/\lambda^2)\|f\|_\lambda$ gives convergence
of the Neumann series. Applying this bound to $\mathsf T|f|$ also
justifies interchanging the sum and the integral. Consequently,
\begin{equation}\label{eq:positive-Green-series}
 \sum_{j=0}^\infty(\mathsf T^{j+1}\mathsf S_{ja}f)(t)
 =\int_{\mathbb R}G_a(u-t)f(u)\dd u.\end{equation}

For $f\in L^1([a,b])$ extended by zero, $\mathsf Tf\in X_\lambda$:
it vanishes for $t>b$ and grows at most linearly as $t\to-\infty$.
Since $\|\mathsf T\mathsf S_a\|_\lambda\le e^{\lambda a}/\lambda^2<1$, the Neumann
series and \eqref{eq:positive-Green-series} give assertion~\ref{item:positive-resolvent}.

For $\Phi\in X_\lambda$, put $h=f-\mathbf1_{[a,b]}\mathsf S_a\Phi$.
This function is integrable and compactly supported, so $\mathsf Th\in X_\lambda$.
By the Neumann series and \eqref{eq:positive-Green-series}, the integral in the statement equals
$(I-\mathsf T\mathsf S_a)^{-1}\mathsf Th$. Therefore
\[
\begin{aligned}
 \Phi=(I-\mathsf T\mathsf S_a)^{-1}\mathsf Th
 &\iff (I-\mathsf T\mathsf S_a)\Phi
       =\mathsf Tf-\mathsf T\mathbf1_{[a,b]}\mathsf S_a\Phi\\
 &\iff \Phi=\mathsf T\mathbf1_{[a,b]^c}\mathsf S_a\Phi+\mathsf Tf.
\end{aligned}
\]
Since
\begin{equation*} \|\mathsf T\mathbf1_{[a,b]^c}\mathsf S_a\|_\lambda
 \le\frac{e^{\lambda a}}{\lambda^2}<1,\end{equation*}
the last equation has the unique solution
$(I-\mathsf T\mathbf1_{[a,b]^c}\mathsf S_a)^{-1}\mathsf Tf$ in $X_\lambda$.

\end{proof}
\subsection{The extremal profile}
\label{subsec:positive-extremal-profile}

The preceding lemma justifies the following definition.
\begin{definition}
\label{def:positive-profile}
For $0<a<2/e$ and $b\ge a$, define
\[
 \Phi_{a,b}:=(I-\mathsf T\mathbf1_{[a,b]^c}\mathsf S_a)^{-1}
 \mathsf T\mathbf1_{[a,b]}
 =\sum_{j=0}^\infty(\mathsf T\mathbf1_{[a,b]^c}\mathsf S_a)^j
 \mathsf T\mathbf1_{[a,b]}.
\]
Equivalently, for every $t\in\mathbb R$,
\begin{equation}\label{eq:positive-excess}
 \Phi_{a,b}(t)=\int_a^bG_a(u-t)[1-\Phi_{a,b}(u-a)]\dd u.\end{equation}
Let $\mathcal E(f):=f(0)$ denote evaluation at zero. We write
\[
 \mathcal E(a,b):=\mathcal E(\Phi_{a,b})=\Phi_{a,b}(0).
\]
\end{definition}

\Needspace{7\baselineskip}

\begin{corollary}
\label{cor:positive-shape}
Let $0<a<2/e$ and $b\ge a$, with $G_a$ as in
Lemma~\ref{lem:positive-operator} and $\Phi_{a,b},\mathcal E(a,b)$ as in
Definition~\ref{def:positive-profile}. Set
\[
 r=-\frac2aW_0(-a/2)>0,\qquad s=\frac2aW_0(a/2)>0,
\]
where $W_0$ is the principal Lambert branch~\cite{Corless}.
If $b=a$, then $\Phi_{a,b}=0$. If $b>a$, then $\Phi_{a,b}$ is strictly positive and strictly decreasing on $\mathbb{R}$ and is a solution of
\begin{equation}  \label{eq:positive-profile-equation}
\Phi_{a,b}''(t)=
\begin{cases}
\Phi_{a,b}(t-a), & t\notin[a,b], \\
1, & a\le t\le b%
\end{cases}.%
\end{equation}
Moreover,
\begin{equation}\label{eq:positive-profile-decay}
 \Phi_{a,b}(t),\ \Phi_{a,b}'(t)=O(e^{-2t/a})
 \qquad\text{as }t\to\infty.
\end{equation}

The function $\Phi_{a,b}$ satisfies
\begin{equation}\label{eq:positive-history}
\begin{aligned}
 \Phi_{a,b}(t)&=A_1e^{-rt}+A_2e^{st},\qquad t\le a,\\
 \Phi_{a,b}(0)&=\mathcal E(a,b)=A_1+A_2.
\end{aligned}\end{equation}

\end{corollary}

\begin{proof}
\proofpart{I}{Regularity and decay}
The defining positive Neumann series gives
$\Phi_{a,b}\ge0$, and $\Phi_{a,a}=0$. If $b>a$, its first term is
\[
 (\mathsf T\mathbf1_{[a,b]})(t)
 =\frac{(b-t)_+^2-(a-t)_+^2}{2}>0\qquad(t<b).
\]
Successive terms of the Neumann series propagate positivity from $t<b$ to $t<b+ja$,
so $\Phi_{a,b}>0$ on $\mathbb R$. Write its operator equation as
$\Phi_{a,b}=\mathsf Tg$, where
\[
 g(t)=\mathbf1_{[a,b]^c}(t)\Phi_{a,b}(t-a)
      +\mathbf1_{[a,b]}(t).
\]
This forcing is locally integrable and positive. Put
$C=\|\Phi_{a,b}\|_{2/a}$. The definition of the norm gives
\[
 |\Phi_{a,b}(t)|\le C e^{-2t/a},\qquad
 g(t)=\Phi_{a,b}(t-a)\le C e^2e^{-2t/a}\quad(t>b).
\]
Thus $\int_t^\infty(1+u-t)g(u)\dd u<\infty$ for every $t$,
so Lemma~\ref{lem:terminal-green} applies to $\mathsf Tg$.
In particular, $\Phi_{a,b}\in C^1(\mathbb R)$, its derivative is locally
absolutely continuous, and $\Phi_{a,b}''=g,$  which is
\eqref{eq:positive-profile-equation}. Differentiating the integral gives
\[
 \Phi_{a,b}'(t)=-\int_t^\infty g(u)\dd u<0.
\]
For $t>b$, the preceding bound on $g$ also yields
\[
 |\Phi_{a,b}'(t)|
 \le C e^2\int_t^\infty e^{-2u/a}\dd u
 =\frac{aCe^2}{2}e^{-2t/a}.
\]
This proves both asserted decay estimates.

\proofpart{II}{Representation}
On $|z|=2/a$, one has $|e^{az}|\le
e^2<4/a^2=|z^2|$. Rouch\'e's theorem shows that the only enclosed zeros of $%
z^2-e^{az}$ are $r$ and $-s$. To obtain the contour representation, note that on $|z|=2/a$,
\[
 \left|\frac{e^{az}}{z^2}\right|\le\left(\frac{ea}{2}\right)^2<1.
\]
Consequently, the geometric expansion
\[
 \frac{e^{zv}}{z^2-e^{az}}
 =\sum_{j=0}^\infty\frac{e^{z(v+ja)}}{z^{2j+2}}
\]
converges uniformly on the circle. Integrating term by term around the
positively oriented circle and applying Cauchy's differentiation formula,
we obtain, for $v\ge0$,
\[
 \begin{aligned}
 \frac{1}{2\pi i}\int_{|z|=2/a}\frac{e^{zv}}{z^2-e^{az}}\dd z
 &=\sum_{j=0}^\infty\frac{1}{2\pi i}
   \int_{|z|=2/a}\frac{e^{z(v+ja)}}{z^{2j+2}}\dd z\\
 &=\sum_{j=0}^\infty\frac{(v+ja)^{2j+1}}{(2j+1)!}
 =G_a(v).
 \end{aligned}
\]
The last equality uses $v+ja\ge0$, so the positive parts in the definition
of $G_a$ can be omitted.
Taking the two residues yields
\begin{equation}  \label{eq:positive-Green-modes}
G_a(v)=\frac{e^{rv}}{r(2-ar)} -\frac{e^{-sv}}{s(2+as)},\qquad v\geq0.\end{equation}

If $t\le a$, then $u-t\ge0$ throughout the integral in
\eqref{eq:positive-excess}. Inserting \eqref{eq:positive-Green-modes}
gives \eqref{eq:positive-history}, with
\begin{equation}\label{eq:positive-amplitudes}
\begin{aligned} A_1&=\frac1{r(2-ar)}\int_a^b
e^{ru}[1-\Phi_{a,b}(u-a)]\dd u,\\ A_2&=-\frac1{s(2+as)}\int_a^b
e^{-su}[1-\Phi_{a,b}(u-a)]\dd u. \end{aligned}\end{equation}
\end{proof}

\begin{corollary}
\label{cor:positive-supersolution}
Let $0<a<2/e$ and $b\ge a$.
Every global positive decreasing
solution $z$ of \eqref{eq:positive-comparison} satisfies
\[
 z(t)\ge z(0)\Phi_{a,b}(t),\qquad t\in\mathbb R.
\]
Such a solution exists if and only if $\mathcal E(a,b)\leq1$.
\end{corollary}

\begin{proof}
For a positive decreasing solution $z$, backward Taylor expansion gives
\[
 z(t)=z(R)-(R-t)z'(R)+\int_t^R(u-t)z''(u)\dd u.
\]
The endpoint terms have nonnegative sum. Dropping them, using
\eqref{eq:positive-comparison}, and letting $R\to\infty$ by monotone
convergence gives
\[
 z\ge \mathsf T\mathbf1_{[a,b]^c}\mathsf S_a z
     +z(0)\mathsf T\mathbf1_{[a,b]}.
\]
Set
\[
 A=\mathsf T\mathbf1_{[a,b]^c}\mathsf S_a,
 \qquad g=\mathsf T\mathbf1_{[a,b]},
 \qquad \mathcal E(f)=f(0).
\]

Let $X$ be the linear span of
\[
 \{A^nz:n\ge0\}\cup\{A^ng:n\ge0\}\cup\{\Phi_{a,b}\},
\]
ordered by the pointwise cone inherited from the ambient function space:
$f\le h$ means $f(t)\le h(t)$ for every $t\in\mathbb R$. With this order,
$A:X\to X$ is positive and $\mathcal E(f)=f(0)$ is a positive linear functional
on $X$.
By Definition~\ref{def:positive-profile},
\[
 \Phi_{a,b}=\sup_{N\ge0}\sum_{j=0}^N A^jg,
 \qquad
 \Phi_{a,b}=A\Phi_{a,b}+g.
\]
The preceding supersolution inequality is
\[
 z\ge Az+\mathcal E(z)g.
\]
Corollary~\ref{cor:rank-one-feedback} therefore gives
\[
 z\ge \mathcal E(z)\Phi_{a,b}=z(0)\Phi_{a,b}.
\]
Applying $\mathcal E$ yields
\[
 z(0)\ge z(0)\mathcal E(a,b).
\]
Since $z(0)>0$, necessarily $\mathcal E(a,b)\le1$.
Conversely, suppose $b>a$ and $\mathcal E(a,b)\le1$. Then $\Phi_{a,b}$
is positive and decreasing and, by
\eqref{eq:positive-profile-equation}, satisfies
\eqref{eq:positive-comparison}, since its forcing on $[a,b]$ is
$1\ge\Phi_{a,b}(0)=\mathcal E(a,b)$. If $b=a$, then $g=0$ and $\mathcal E(a,a)=0$.
In this case $z(t)=e^{-2t/a}$ is a positive decreasing solution because
$a<2/e$ gives
\[
 z''(t)=\frac4{a^2}z(t)>e^2z(t)=z(t-a).\qedhere
\]
\end{proof}

\subsection{The critical curve
\texorpdfstring{$b=\beta(a)$}{b=beta(a)}}
\label{subsec:positive-critical-curve}

\begin{proposition}
\label{prop:positive-crossing} For each $0<a<2/e$ there is a unique $%
\beta(a)\in(a,\sqrt2)$ with
\begin{equation*}  %
\mathcal E(a,\beta(a))=1, \qquad \mathcal E(a,b)\leq1\ \Longleftrightarrow\ b\le\beta(a).\end{equation*}
\end{proposition}

\begin{proof}
Fix $0<a<2/e$. For $b_2>b_1\ge a$,
\begin{equation}  \label{eq:positive-difference}
\Phi_{a,b_2}-\Phi_{a,b_1} =(I-\mathsf T\mathbf1_{[a,b_2]^c}\mathsf S_a )^{-1} \left[t\mapsto\int_t^%
\infty(u-t)\mathbf{1}_{[b_1,b_2]}(u) (1-\Phi_{a,b_1}(u-a))\dd u\right]%
.\end{equation}
Suppose $\mathcal E(a,b_1)=\Phi_{a,b_1}(0)\le1$. Then
$1-\Phi_{a,b_1}(u-a)\ge0$ for $b_1\le u\le b_2$, with strict
inequality for $b_1<u<b_2$. This follows from strict decrease when
$b_1>a$, while $\Phi_{a,a}=0$ handles $b_1=a$.
The function in brackets in \eqref{eq:positive-difference} is therefore
nonnegative and strictly positive at zero. The positive Neumann expansion
of the inverse operator, whose first term is the identity, gives
$\mathcal E(a,b_2)>\mathcal E(a,b_1)$. Thus, for fixed $a$, the function $b\mapsto \mathcal E(a,b)$
strictly increases until it reaches one.

For fixed $a$, the integral kernels and the Neumann series also show that
$b\mapsto \mathcal E(a,b)$ is continuous. Since $\mathcal E(a,a)=0$, it remains to prove
that this function reaches one. If $b>a$ and $\mathcal E(a,b)\le1$, the forcing in
\eqref{eq:positive-profile-equation} is at least $\mathcal E(a,b)$ on $[0,b]$.
Backward Taylor expansion gives
\[
 \mathcal E(a,b)\ge\Phi_{a,b}(b)+\frac{b^2}{2}\mathcal E(a,b)
 >\frac{b^2}{2}\mathcal E(a,b),
\]
so $b<\sqrt2$. Hence $\mathcal E(a,\sqrt2)>1$, proving the assertion.
\end{proof}

\begin{lemma}
\label{lem:positive-short} Let $0<a<2/e$ and $a\le b\leq2a$. Set $\ell=b-a$
and
\begin{equation*}I(z,\ell)=%
\begin{cases}
(e^{z\ell}-1)/z, & z\ne0, \\
\ell, & z=0,%
\end{cases}
\qquad c_+=\frac r{2-ar},\qquad c_-=-\frac s{2+as}. \end{equation*}
\Needspace{9\baselineskip}
The amplitude vector $(A_1,A_2)^{\top}$ in Corollary~\ref{cor:positive-shape} is the unique solution of
\begin{equation}  \label{eq:positive-matrix}
\begin{pmatrix}
1+c_+\ell & c_+I(r+s,\ell) \\
c_-I(-r-s,\ell) & 1+c_-\ell%
\end{pmatrix}
\begin{pmatrix}
A_1 \\
A_2%
\end{pmatrix}
=%
\begin{pmatrix}
c_+I(r,\ell) \\
c_-I(-s,\ell)%
\end{pmatrix}%
.\end{equation}
The equality $b=\beta(a)$ is characterized by
\begin{equation*}  %
A_1+A_2=1.\end{equation*}
Define the scalar function
\begin{equation*}\begin{aligned} \Psi_a(\ell)={}& [1+c_+(\ell-I(r,\ell))]
[1+c_-(\ell-I(-s,\ell))]\\ &-c_+c_- [I(r+s,\ell)-I(r,\ell)]
[I(-r-s,\ell)-I(-s,\ell)]. \end{aligned}\end{equation*}
If $\beta(a)\le2a$, then
\begin{equation*}
\beta(a)=a+\ell_a,\qquad \Psi_a(\ell_a)=0,\qquad 0<\ell_a\le a,
\end{equation*}
and this root is unique.
\end{lemma}

\begin{proof}
Set $v=u-a$. Since $0\le v\le\ell\le a$, the history formula
\eqref{eq:positive-history} gives $\Phi_{a,b}(v)=A_1e^{-rv}+A_2e^{sv}$.
Substitution into \eqref{eq:positive-amplitudes} yields
\[
\begin{aligned}
 A_1&=c_+\bigl[I(r,\ell)-A_1\ell-A_2I(r+s,\ell)\bigr],\\
 A_2&=c_-\bigl[I(-s,\ell)-A_1I(-r-s,\ell)-A_2\ell\bigr],
\end{aligned}
\]
which is \eqref{eq:positive-matrix}.

To prove invertibility, suppose $\ell>0$ (the case $\ell=0$ is trivial) and $(B_1,B_2)$ is a nonzero vector in the
kernel of the square matrix in \eqref{eq:positive-matrix}. Define $h(v)=B_1e^{-rv}+B_2e^{sv}$, and define
\[
 \psi(t)=-\int_a^bG_a(u-t)h(u-a)\dd u.
\]
The homogeneous version of \eqref{eq:positive-matrix} and \eqref{eq:positive-Green-modes}
give $\psi(t)=h(t)$ for $t\le a$, so $\psi\not\equiv0$.
By Lemma~\ref{lem:positive-operator}\textup{\ref{item:positive-resolvent}}, applied to the forcing
$u\mapsto-\mathbf1_{[a,b]}(u)h(u-a)$, we have $\psi\in X_{2/a}$. Since $u-a\le a$ on $[a,b]$, we may replace
$h(u-a)$ by $\psi(u-a)$ in its defining integral. The homogeneous case of Lemma~\ref{lem:positive-operator}
therefore gives $\psi=0$, a contradiction. Thus the matrix is invertible.

Let $\mathsf{M}=\mathsf{M}(a,\ell)$ denote the coefficient matrix in~\eqref{eq:positive-matrix}, and let $\mathbf{d}=\mathbf{d}(a,\ell)$ denote its right-hand-side vector.
Direct expansion gives the first equality below, and taking $\mathsf{M}$ as a common factor gives
\[
\begin{aligned}
 \Psi_a(\ell)
 &=\det\!\bigl(\mathsf{M}-\mathbf{d}(1,1)\bigr)\\
 &=\det(\mathsf{M})\bigl[1-(1,1)\mathsf{M}^{-1}\mathbf{d}\bigr]\\
 &=\det(\mathsf{M})\bigl[1-\mathcal E(a,a+\ell)\bigr],
\end{aligned}
\]
since $\mathsf{M}^{-1}\mathbf{d}=(A_1,A_2)^{\top}$ and
$A_1+A_2=\mathcal E(a,a+\ell)$. Because $\det(\mathsf{M})\ne0$,
\[
 \Psi_a(\ell)=0
 \iff \mathcal E(a,a+\ell)=1
 \iff a+\ell=\beta(a).
\]
The last equivalence follows from Proposition~\ref{prop:positive-crossing},
which establishes the unique crossing $\mathcal E(a,\beta(a))=1$.
Thus, if $\beta(a)\le2a$, the unique root in $0<\ell\le a$ is
$\ell=\beta(a)-a$.
\end{proof}

\begin{lemma}
Let $0<a<2/e$ and put
\[
 d=\frac1s-\frac1r.
\]
For every $b\ge2a$,
\begin{equation}\label{eq:positive-long-profile}
 \mathcal E(a,b)=\frac{rs}{2}(b-2a+2d)^2+rs(1-d^2)-1.\end{equation}
Consequently, $\beta(a)\ge2a$ exactly when $rs(1+d^2)\le2$. In that case
\begin{equation}\label{eq:positive-explicit}
 \beta(a)=2a-2d+\sqrt{2d^2+\frac4{rs}-2}.\end{equation}
\end{lemma}

\begin{proof}
Write $\Phi=\Phi_{a,b}$. For $z\in\{r,-s\}$, multiply
$\Phi''(t)=\Phi(t-a)$ on $[b,\infty)$ by $e^{z(t-b)}$ and integrate.
The decay in Corollary~\ref{cor:positive-shape} justifies integration by parts at infinity. Thus
\[
 \int_b^\infty e^{z(t-b)}\Phi''(t)\dd t
 =z\Phi(b)-\Phi'(b)
 +z^2\int_b^\infty e^{z(t-b)}\Phi(t)\dd t.
\]
Using $\Phi''(t)=\Phi(t-a)$, the change of variables $u=t-a$, and $z^2=e^{az}$, the common
contribution from $[b,\infty)$ cancels, giving
\begin{equation*} z\Phi(b)-\Phi'(b)
 =z^2\int_{b-a}^b e^{z(u-b)}\Phi(u)\dd u.\end{equation*}
Since $b\ge2a$, one has $\Phi''=1$ on $[b-a,b]$. Evaluating the last
identity for $z=r,-s$ yields
\[
 \Phi'(b)=a-d,
 \qquad
 \Phi(b)=1+\frac{a^2}{2}-ad-\frac1{rs}.
\]
Again using $\Phi''=1$ on $[a,b]$,
\[
 \Phi'(a)=2a-b-d,
 \qquad
 \Phi(a)=\Phi(b)-(a-d)(b-a)+\frac{(b-a)^2}{2}.
\]
The history representation \eqref{eq:positive-history}, together with
$e^{-ra}=r^{-2}$ and $e^{sa}=s^{-2}$, gives
\[
 \mathcal E(a,b)=rs\Phi(a)+(s-r)\Phi'(a).
\]
This proves \eqref{eq:positive-long-profile}.
At $b=2a$ this gives $\mathcal E(a,2a)=rs(1+d^2)-1$, so
Proposition~\ref{prop:positive-crossing} yields
$\beta(a)\ge2a$ exactly when $rs(1+d^2)\le2$. In this case,
$\mathcal E(a,\beta(a))=1$ and \eqref{eq:positive-long-profile} give
\[
 (\beta(a)-2a+2d)^2=2d^2+\frac4{rs}-2,
\]
and the positive root is \eqref{eq:positive-explicit}.
\end{proof}

\begin{lemma}
\label{lem:positive-boundary} The boundary $\beta$ extends to a continuous
strictly decreasing function on $[0,2/e]$, with $\beta(0)=\sqrt2$ and $%
\beta(2/e)=2/e$.
\end{lemma}

\begin{proof}
\proofpart{I}{Continuity on $(0,2/e)$}
Fix $a_*\in(0,2/e)$ and use the common weight $\lambda=2/a_*$.
For $a$ near $a_*$, the bound $e^{\lambda a}/\lambda^2\le q<1$
holds with fixed $\lambda$ and $q$. The weighted integral operators of
Lemma~\ref{lem:positive-operator} depend continuously on $a,b$ in operator
norm, and $\mathsf T\mathbf1_{[a,b]}$ depends continuously on them in
$X_\lambda$. The uniformly convergent Neumann series therefore makes
$\mathcal E(a,b)$ jointly continuous. Set $b_*=\beta(a_*)$. For every sufficiently
small $\varepsilon>0$, Proposition~\ref{prop:positive-crossing} gives
\[
 \mathcal E(a_*,b_*-\varepsilon)<1<\mathcal E(a_*,b_*+\varepsilon).
\]
By continuity these strict inequalities persist for $a$ near $a_*$.
The same proposition then gives
$b_*-\varepsilon<\beta(a)<b_*+\varepsilon$, proving continuity of
$\beta$ on $(0,2/e)$.

\proofpart{II}{Strict monotonicity}
Fix $0<a_1<a_2<2/e$, set $b_2=\beta(a_2)$, and set
$X=\Phi_{a_2,b_2}$. Then $X(0)=1$, $X$ is strictly decreasing, and
\[
 X''(t)=\mathbf1_{[a_2,b_2]^c}(t)X(t-a_2)
       +\mathbf1_{[a_2,b_2]}(t).
\]
Define
\[
 \delta(t)=X''(t)-\mathbf1_{[a_1,b_2]^c}(t)X(t-a_1)
             -\mathbf1_{[a_1,b_2]}(t).
\]
Since $X$ is strictly decreasing and $X(0)=1$, we have almost everywhere
\[
\begin{aligned}
 \delta(t)
 &=\mathbf1_{[a_1,b_2]^c}(t)
   \bigl[X(t-a_2)-X(t-a_1)\bigr]\\
 &\quad+\mathbf1_{[a_1,a_2]}(t)\bigl[X(t-a_2)-1\bigr]
 \ge0.
\end{aligned}
\]
In particular, $\delta(t)>0$ almost everywhere on $(b_2,\infty)$.
Since $X$ and $X'$ decay exponentially, two integrations give
\[
 X=\mathsf T\mathbf1_{[a_1,b_2]^c}\mathsf S_{a_1}X+\mathsf T\mathbf1_{[a_1,b_2]}+h,
 \qquad h(t)=\int_t^\infty(u-t)\delta(u)\dd u.
\]
Here $h\ge0$ and $h(0)>0$. Bootstrapping and discarding nonnegative terms yield
\[
 X\ge\sum_{j=0}^N(\mathsf T\mathbf1_{[a_1,b_2]^c}\mathsf S_{a_1})^{\,j}\mathsf T\mathbf1_{[a_1,b_2]}+h,
 \qquad N\ge0.
\]
Letting $N\to\infty$ and evaluating at zero, we obtain
\[
 1=X(0)\ge \mathcal E(a_1,b_2)+h(0)>\mathcal E(a_1,b_2).
\]
Proposition~\ref{prop:positive-crossing} therefore gives
$b_2<\beta(a_1)$, so $\beta(a_2)<\beta(a_1)$.

\proofpart{III}{Continuity at $0$}
The formulas in Corollary~\ref{cor:positive-shape} and
$W_0(z)/z\to1$ as $z\to0$ give $r,s\to1$ as $a\to0^{+}$.
At $a=0$, fix $\lambda>1$. Since $\mathsf S_0=I$ and
$\|\mathsf T\mathbf1_{[0,b]^c}\|_\lambda\le\lambda^{-2}<1$, the same
positive Neumann series defines $\Phi_{0,b}$. Its first term
$\mathsf T\mathbf1_{[0,b]}$ vanishes for $t\ge b$ and equals
$(b-t)^2/2$ on $[0,b]$. All subsequent terms vanish for $t\ge0$:
the integral only uses $u\ge t\ge0$, the indicator removes $[0,b]$,
and the preceding term vanishes for $u\ge b$. Consequently,
$\mathcal E(0,b)=b^2/2$.

For sufficiently small $a\ge0$, the bound
$e^{\lambda a}/\lambda^2\le q<1$ holds with fixed $\lambda$ and $q$.
As in Part~I, continuity of the weighted integral operators and their
resolvents gives continuity of $\mathcal E$ up to $a=0$. For any sufficiently
small $\varepsilon>0$,
\[
 \mathcal E(0,\sqrt2-\varepsilon)<1<\mathcal E(0,\sqrt2+\varepsilon).
\]
These inequalities persist for sufficiently small positive $a$.
Applying Proposition~\ref{prop:positive-crossing} at those positive values
of $a$ gives
$\sqrt2-\varepsilon<\beta(a)<\sqrt2+\varepsilon$. Thus
\[
 \lim_{a\to0^{+}}\beta(a)=\sqrt2,
\]
and defining $\beta(0)=\sqrt2$ gives the continuous extension.

\proofpart{IV}{Continuity at $2/e$}
For the other endpoint, let $X=\Phi_{a,\beta(a)}$. Since $X(0)=1$
and the forcing is at least one on $[0,\beta(a)]$, we have
\[
 -X'(u)=\int_u^\infty X''(t)\dd t
 \ge\int_u^{\beta(a)}1\dd t
 =\beta(a)-u,
 \qquad 0\le u\le\beta(a).
\]
Consequently, for $0\le v\le\beta(a)$,
\[
 1-X(v)=\int_0^v[-X'(u)]\dd u
 \ge\beta(a)v-\frac{v^2}{2}
 \ge\frac{v^2}{2}.
\]
Equation \eqref{eq:positive-excess} at zero therefore gives
\begin{equation}  \label{eq:positive-critical-integral}
1=\int_0^{\beta(a)-a}G_a(a+v)[1-X(v)]\dd v \ge\int_0^{\beta(a)-a}G_a(a+v)\frac{%
v^2}{2}\dd v.\end{equation}
By the formulas in Corollary~\ref{cor:positive-shape}, continuity of $W_0$
and $W_0(-1/e)=-1$ give, as $a\to(2/e)^{-}$,
\[
 r\to e^{-},\qquad 2-ar\to0^{+},\qquad
 s\to eW_0(1/e)>0,\qquad 2+as\to2+2W_0(1/e)>0.
\] By \eqref{eq:positive-Green-modes}, $%
G_a(a+v)\to\infty$ uniformly on every compact subinterval of $(0,\infty)$. If
$\beta(a)-a$ were bounded below by $\varepsilon>0$ along a sequence, the integral
on $[\varepsilon/2,\varepsilon]$ in \eqref{eq:positive-critical-integral}
would diverge. Thus $\beta(a)-a\to0$ and
\begin{equation*}\lim_{a\to(2/e)^{-}}\beta(a)=2/e.\end{equation*}
\end{proof}

\begin{remark}
By Lemma~\ref{lem:positive-boundary}, there is a unique
$a_{\mathrm c}\in(0,2/e)$ such that
\[
 \beta(a_{\mathrm c})=2a_{\mathrm c}.
\]
For $0<a\le a_{\mathrm c}$, the explicit formula
\eqref{eq:positive-explicit} applies. For $a_{\mathrm c}<a<2/e$, one has
$\beta(a)=a+\ell_a$, where $\ell_a\in(0,a)$ is the unique root of the
scalar equation $\Psi_a(\ell)=0$ in Lemma~\ref{lem:positive-short}.
The transition is characterized by $rs(1+d^2)=2$, with
$d=1/s-1/r$ and $r,s$ corresponding to $a=a_{\mathrm c}$.
Numerically,
\[
 a_{\mathrm c}\approx0.6307116580,
 \qquad
 \beta(a_{\mathrm c})=2a_{\mathrm c}\approx1.2614233161.
\]
\end{remark}

\subsection{Proof of the oscillation bound and sharpness}
\label{subsec:positive-proof-sharpness}

We now prove the main results of this section. Assuming the existence of an eventually positive decreasing solution, one can construct a global positive decreasing solution $z$. The comparison between $z$ and $\Phi$, as in Corollaries~\ref{cor:rank-one-feedback} and~\ref{cor:positive-supersolution}, then follows directly. However, such a global construction would require further asymptotic estimates, of the type we shall later develop in Section~\ref{sec:spectral-positive}. Here we instead give a self-contained proof, showing that these refined asymptotic estimates are not needed for the oscillation bound. As for sharpness, only a mild spectral separation is required (see the proof of Theorem~\ref{thm:positive-sharpness} below).

\begin{proof}[Proof of Theorem~\ref{thm:positive-main}]
Suppose an eventually positive decreasing solution exists. The classification in~\cite{Kamenskii,Ladde},
 gives $x,x'\to0$. Two integrations and the Fubini-Tonelli theorem yield
\begin{equation}\label{eq:positive-final-integral}
 x(t)=\int_t^\infty(u-t)x(\tau(u))\dd u.\end{equation}
The case $a=0$ follows from the classical \(\limsup\) criterion
\cite{Gustafson}, \cite{Ladde} (see
Proposition~\ref{prop:BDS16}). Proposition~\ref{prop:autonomous-liminf}
excludes $a>2/e$. Suppose therefore that $0<a\le2/e$ and, toward a
contradiction, that $b>\beta(a)$. By continuity of $\beta$, choose
\[
 0<a_0<a,\qquad a_0<b_0<b,
 \qquad b_0>\beta(a_0).
\]
Since $\mathcal E(a_0,b_0)>1$ and the defining Neumann series for
$\Phi_{a_0,b_0}$ is increasing, some finite partial sum already exceeds one:
\begin{equation}\label{eq:positive-finite}
 \sum_{j=0}^N
 \bigl[(\mathsf T\mathbf1_{[a_0,b_0]^c}\mathsf S_{a_0})^{\,j}
 \mathsf T\mathbf1_{[a_0,b_0]}\bigr](0)>1.\end{equation}
Fix such an $N$. Choose $t_*$ so that $x$ is positive and decreasing
on $[t_*,\infty)$ and
\[
 t-\tau(t)\ge a_0\qquad(t\ge t_*).
\]
Choose a time $T$ with $T-\tau(T)\ge b_0$, and choose $T$ sufficiently
late that
\[
 \tau(T)-(N+1)a_0\ge t_*,\qquad
 \tau\bigl(\tau(T)-Na_0\bigr)\ge t_*.
\]
Define
\[
 z(t)=\frac{x(\tau(T)+t)}{x(\tau(T))},
 \qquad t\ge t_*-\tau(T),
 \qquad z(0)=1.
\]
For $u\ge-Na_0$, monotonicity of $x$ gives
\[
 \frac{x(\tau(\tau(T)+u))}{x(\tau(T))}\ge z(u-a_0).
\]
For $u\in[a_0,b_0]$, monotonicity gives
$\tau(\tau(T)+u)\le\tau(T)$, so this ratio is at least one. Thus
\eqref{eq:positive-final-integral} gives, for $t\ge-Na_0$,
\begin{equation}\label{eq:positive-finite-comparison}
 z(t)\ge g(t)+\int_t^\infty(u-t)\mathbf1_{[a_0,b_0]^c}(u)
 z(u-a_0)\dd u,
 \qquad g=\mathsf T\mathbf1_{[a_0,b_0]}.\end{equation}
Define inductively
\[
 v_0=g,\qquad
 v_{k+1}(t)=g(t)+\int_t^\infty(u-t)\mathbf1_{[a_0,b_0]^c}(u)
 v_k(u-a_0)\dd u.
\]
Positivity first gives $z\ge v_0$ on $[-Na_0,\infty)$. If
$z\ge v_k$ on $[-(N-k)a_0,\infty)$, substitution into
\eqref{eq:positive-finite-comparison} gives
$z(t)\ge v_{k+1}(t)$. Bootstrapping yields
\[
 z(t)\ge v_k(t)\quad\text{for }t\ge-(N-k)a_0,
 \qquad 0\le k\le N.
\]
Evaluating at zero and using \eqref{eq:positive-finite}, we obtain
\[
 1=z(0)\ge v_N(0)
 =\sum_{j=0}^N
 \bigl[(\mathsf T\mathbf1_{[a_0,b_0]^c}\mathsf S_{a_0})^j
 \mathsf T\mathbf1_{[a_0,b_0]}\bigr](0)>1,
\]
a contradiction. Therefore $b\le\beta(a)$.
\end{proof}

The following sharpness proof uses a gliding-hump construction of multipulse type, cf.\ \cite{BanachSteinhaus1927,DickmeisNessel1981,Feroe1981,Lebesgue1905,Sandstede1998,ZelikMielke2009}.

\begin{theorem}
\label{thm:positive-sharpness}
For every $a\in[0,2/e]$ and $b\in[a,\beta(a)]$, there is a nondecreasing
argument $\tau$ with
$\tau(t)\le t$ and $\tau(t)\overset{t\to\infty}{\to}\infty$ for which
\[
 x''(t)=x(\tau(t))
\]
admits a strictly positive decreasing solution with
\[
 \liminf_{t\to\infty}(t-\tau(t))=a,\qquad
 \limsup_{t\to\infty}(t-\tau(t))=b.
\]
\end{theorem}

\begin{proof}
If $a=2/e$, then $b=2/e$ and one may take $\tau(t)=t-2/e$ and
$x(t)=e^{-et}$. Thus assume $0\le a<2/e$. By Proposition~\ref{prop:BDS17}, it is enough to prove the endpoint case
$b=\beta(a)$.
Fix $b=\beta(a)$, and set
\[
 b_j=a+(b-a)(1-2^{-j}),\qquad I_j=[T_j+a,T_j+b_j],
\]
where the times $T_j$ will be chosen below. Prescribe
\begin{equation}\label{eq:sharp-plateau-delay}
 \tau(t)=\begin{cases}
 T_j,&t\in I_j,\\
 t-a,&t\notin\bigcup_{j\ge1}I_j
 \end{cases}.\end{equation}

\proofpart{I}{Construction and properties of $u_j$}
If $a>0$, let $u_j=\Phi_{a,b_j}$, with $\Phi_{a,b_j}$ as in
Definition~\ref{def:positive-profile}. Thus $u_j$ is positive and decreasing.
For $a=0$, use instead the following explicit profiles with $b_j<\sqrt2$:
\begin{equation}\label{eq:sharp-plateau-zero}
 u_j(s)=\begin{cases}
 \dfrac{b_j^2}{2}\cosh s-b_j\sinh s,&s\le0,\\
 \dfrac{(b_j-s)^2}{2},&0\le s\le b_j,\\
 0,&s\ge b_j
 \end{cases}.\end{equation}
In both cases $u_j$ is nonnegative and nonincreasing.
Define
\[
 \varepsilon_j:=1-u_j(0)>0.
\]
These profiles satisfy,
\begin{equation*} u_j''(s)=\begin{cases}
 1,&a\le s\le b_j,\\
 u_j(s-a),&s\notin[a,b_j]
 \end{cases}.\end{equation*}

There exist constants $k,\delta>0,C>1$, independent of $j$, such that
\begin{equation}\label{eq:sharp-plateau-envelope}
 e^{ks}\bigl(u_j(s)+|u_j'(s)|+u_j''(s)\bigr)
 \le C e^{-\delta|s|}\quad \forall s\in\mathbb R.\end{equation}
Here one can take $r<k<2/a$ for $a>0$, with $r$ the smaller positive
root of $r^2=e^{ar}$. For $a=0$ take $k>1$.

The comparison in \eqref{eq:positive-difference} gives
$0\le u_j\le\Phi_{a,b}$. Also $u_j''\le\Phi_{a,b}''$ (because for $s\in(b_j,b]$, we have $u_j(s-a)\le\Phi_{a,b}(s-a)\le1$).
Since $-u_j'(s)=\int_s^\infty u_j''(v)\dd v$, the derivatives are similarly dominated. The two-mode history \eqref{eq:positive-history} and the weighted
bound \eqref{eq:positive-profile-decay} therefore give \eqref{eq:sharp-plateau-envelope}. For $a=0$ the same conclusion follows directly
from \eqref{eq:sharp-plateau-zero}. This justifies \eqref{eq:sharp-plateau-envelope}.

\proofpart{II}{Gliding-hump supersolution}
Set $T_1=1$, and recursively define, for $j\ge2$,
\begin{equation*} T_j=1+\max\left\{
 T_{j-1}+b+1,\quad
 \max_{1\le i<j}\left[
 T_i+\frac1\delta\log\frac{C\,2^{i+j+1}}
 {\varepsilon_i\varepsilon_j}\right]\right\}.\end{equation*}
In particular the intervals $I_j$ are disjoint, and for every $i\ne j$,
\begin{equation*}%
 C e^{-\delta|T_i-T_j|}
 \le\varepsilon_i\varepsilon_j2^{-i-j-1}.\end{equation*}
Consequently,
\begin{equation}\label{eq:sharp-plateau-rows}
 C\sum_{i\ne j}e^{-\delta|T_i-T_j|}
 \le\varepsilon_j2^{-j-1}\sum_{i\ne j}\varepsilon_i2^{-i}
 \le\frac{\varepsilon_j}{2},\end{equation}
since $0<\varepsilon_i\le1$.

Take $c_j=e^{-kT_j}$ and
\begin{equation*}%
 Z(t)=\sum_{j\ge1}c_ju_j(t-T_j).\end{equation*}
Estimates \eqref{eq:sharp-plateau-envelope} and \eqref{eq:sharp-plateau-rows} ensure locally uniform convergence of the
function and first-derivative series and locally dominated convergence
of the second-derivative series. The definition of $u_j$ ensures that $Z>0$ and $Z'<0$.
Moreover,
\[
 e^{kt}Z(t)\le C\sum_j e^{-\delta|t-T_j|}\le C_1,
\]
where the last bound follows from the uniform separation of the
$T_j$. Thus $Z$ is bounded on $[0,\infty)$ and tends to zero.

We claim that $Z$ and the delay \eqref{eq:sharp-plateau-delay} satisfy
\begin{equation}\label{eq:sharp-plateau-supersolution}
 Z''(t)\ge Z(\tau(t)).\end{equation}
Outside the union of the $I_j$, every summand satisfies its constant-delay
equation, so equality holds. If $t\in I_j$, the only $u_i$ on its
exceptional interval is $u_j$, and
\begin{align*} Z''(t)-Z(T_j)
 &=c_j(1-u_j(0))
   +\sum_{i\ne j}c_i\bigl[u_i(t-T_i-a)-u_i(T_j-T_i)\bigr]\\
 &\ge c_j\varepsilon_j-\sum_{i\ne j}c_iu_i(T_j-T_i).\end{align*}
By \eqref{eq:sharp-plateau-envelope},
\[
 c_iu_i(T_j-T_i)
 \le c_j C e^{-\delta|T_i-T_j|}.
\]
Using \eqref{eq:sharp-plateau-rows}, we conclude that
\[
 Z''(t)-Z(T_j)\ge\tfrac12c_j\varepsilon_j>0.
\]
This proves \eqref{eq:sharp-plateau-supersolution}.

Applying Proposition~\ref{prop:BDS17} to the positive nonincreasing supersolution $Z$
completes the proof.
\end{proof}

\begin{theorem}
For every $0\le a\le2/e$ and $b\ge2/e$, there exists a nondecreasing
argument $\tau(t)\le t$, tending to infinity, such that
\[
 \liminf_{t\to\infty}(t-\tau(t))=a,\qquad
 \limsup_{t\to\infty}(t-\tau(t))=b,
\]
and $x''(t)=x(\tau(t))$ has no eventually positive nonincreasing solution.
\end{theorem}

\begin{proof}
By Proposition~\ref{prop:autonomous-liminf}, for every $d>2/e$ there is a
finite length $L(d)$ such that no solution of
\[
 y''(t)=y(t-d)
\]
can remain positive and nonincreasing throughout an interval of length
$L(d)$ together with its preceding history of length $d$.

Take $b_j=b+1/j>2/e$ and choose these lengths $L_j=L(b_j)$.
Set $S_1=1$, $E_j=S_j+b_j-a+L_j$, and $S_{j+1}=E_j+1$.
Define
\[
 \tau(t)=
 \begin{cases}
 S_j-a,& S_j\le t<S_j+b_j-a,\\
 t-b_j,& S_j+b_j-a\le t<E_j,\\
 t-a,& t\notin\displaystyle\bigcup_{j\ge1}[S_j,E_j).
 \end{cases}
\]
The resulting argument $\tau$ is nondecreasing: it is constant on the first
piece of each block, increases with slope one on the second piece and between
blocks, and its only jumps are upward.
If an eventually positive nonincreasing solution existed, choose $j$ so large
that $x$ is positive and nonincreasing on $[S_j-a,E_j]$. On the interval
$[S_j+b_j-a,E_j]$, which has length $L_j$, the equation is
\[
 x''(t)=x(t-b_j),
\]
and its preceding history of length $b_j$ is positive and nonincreasing.
This contradicts the defining property of $L_j=L(b_j)$.
\end{proof}

\section{Spectral analysis of the oscillation bound}
\label{sec:spectral-positive}

\subsection{Description of positive decreasing solutions}
\label{subsec:spectral-description}

\begin{remark}\label{rem:spectral-first-order-reduction}
Consider the nonhomogeneous second-order equation
\[
 x''(t)=x(t-a)+f(t),\qquad 0<a<2/e.
\]
Let $\Delta(\lambda)=\lambda^2-e^{-a\lambda}$. Its positive real root $s$
and largest negative real root $-r$ satisfy
\[
 s^2=e^{-as},\qquad r^2=e^{ar},\qquad 0<s<r<2/a.
\]
For a characteristic root $\lambda$, define
\begin{equation*} J_\lambda[x](t):=x'(t)+\lambda x(t)
 +e^{-a\lambda}\int_{-a}^0e^{-\lambda\theta}x(t+\theta)\dd\theta.\end{equation*}
Differentiation of the integral, using $e^{-a\lambda}=\lambda^2$, gives
\begin{equation}\label{eq:PSS-J-forced}
 J_\lambda'[x]-\lambda J_\lambda[x]=x''-x(\cdot-a)=f,\end{equation}
and hence
\begin{equation}\label{eq:PSS-J-transport}
 e^{-\lambda t_1}J_\lambda[x](t_1)
 =e^{-\lambda t_0}J_\lambda[x](t_0)
  +\int_{t_0}^{t_1}e^{-\lambda u}f(u)\dd u.\end{equation}
For $f=0$, we recover the autonomous equation
\begin{equation}\label{eq:spectral-autonomous-positive}
 x''(t)=x(t-a),\qquad 0<a<2/e,\end{equation}
and \eqref{eq:PSS-J-transport} reduces to
\begin{equation}\label{eq:J-evolution}
 J_\lambda[x](t)=e^{\lambda t}J_\lambda[x](0).\end{equation}

In particular, the invariant hyperplane
\[
 \Gamma_s:=\{\varphi\in C^1([-a,0]):
 \varphi'(0)+s\varphi(0)+s^2\int_{-a}^0e^{-s\theta}\varphi(\theta)\dd\theta=0\}
\]
consists exactly of the initial histories whose positive growing-mode
coefficient $c_0=J_s[x](0)/\Delta'(s)$ vanishes.
It is the spectral complement of the growing mode $e^{st}$.
The dynamics on this hyperplane are governed by the first-order delay
equation $J_s[x]=0$, namely
\begin{equation}\label{eq:spectral-first-order-reduction}
 x'(t)+sx(t)+s^2\int_{-a}^0e^{-s\theta}x(t+\theta)\dd\theta=0.
\end{equation}
Indeed, \eqref{eq:J-evolution} preserves $J_s[x]=0$, and
\eqref{eq:PSS-J-forced} shows that every solution of this first-order
equation solves \eqref{eq:spectral-autonomous-positive}.
\end{remark}

To use this reduction to characterize positive decreasing solutions, we first establish a comparison lemma for positive integral kernels.

\begin{lemma}
\label{lem:positive-kernel-comparison}
Let $a>0$, let $k\in L^1([0,a])$ be nonnegative, and put
\[
 (\mathcal Kx)(t):=\int_0^a k(v)x(t-v)\dd v,
 \qquad m:=\int_0^a k(v)\dd v.
\]
\begin{enumerate}[label=\textup{(\roman*)},leftmargin=*,itemsep=\medskipamount]
\item\label{item:kernel-strict}
If $k\not\equiv0$, $x\in C([T-a,\infty))$, $x>0$ on $[T-a,T]$, and
$x\ge\mathcal Kx$ on $[T,\infty)$, then $x>0$ throughout $[T,\infty)$.

\item\label{item:kernel-ae}
If $x$ is locally essentially bounded, $x\ge0$ almost everywhere
on $[T-a,T]$, and $x\ge\mathcal Kx$ almost everywhere on $[T,\infty)$,
then $x\ge0$ almost everywhere there.

\item\label{item:kernel-resolvent}
Assume $m<1$, and let $z\in C([T-a,\infty))$ satisfy
\[
 z(t)=c+\mathcal Kz(t),\qquad t\ge T.
\]
Extend $k$ by zero to $[0,\infty)$ and define its renewal resolvent kernel
\[
 \mathfrak r_k:=\sum_{n=1}^\infty k^{*n}\in L^1(0,\infty),
 \qquad
 \int_0^\infty \mathfrak r_k(u)\dd u=\frac{m}{1-m}.
\]
Writing $\varphi(\theta)=z(T+\theta)$ for $-a\le\theta\le0$, set
\[
 g(u):=c+\mathbf1_{[0,a]}(u)\int_u^a k(v)\varphi(u-v)\dd v,
 \qquad u\ge0.
\]
Then, for every $u\ge0$,
\begin{equation}\label{eq:kernel-resolvent-formula}
 z(T+u)=g(u)+(\mathfrak r_k*g)(u)
 =g(u)+\int_0^u\mathfrak r_k(v)g(u-v)\dd v.\end{equation}
In particular,
\begin{equation}\label{eq:kernel-resolvent-limit}
 \lim_{t\to\infty}z(t)=\frac{c}{1-m}.\end{equation}

\item\label{item:kernel-obstruction}
If $m<1$ and a continuous positive function $x$ on $[T-a,\infty)$ satisfies
$x\le c+\mathcal Kx$ on $[T,\infty)$ for a constant $c$, then $c\ge0$.
\end{enumerate}
\end{lemma}

\begin{proof}
Part~\ref{item:kernel-strict} is the standard first-zero argument:
if $t_0\ge T$ were the first zero of $x$, then
$x(t_0-v)>0$ for $0<v\le a$, and therefore
\[
 (\mathcal Kx)(t_0)>0,
\]
contradicting $x(t_0)\ge(\mathcal Kx)(t_0)$.

For~\ref{item:kernel-ae}, apply the Volterra--Gr\"onwall inequality
to the negative part $x_-$. Indeed,
\[
 0\le x_-\le\mathcal Kx_-,
\]
while $x_-=0$ almost everywhere on $[T-a,T]$. Hence
$x_-=0$ almost everywhere on $[T,\infty)$.
For~\ref{item:kernel-resolvent}, set $Z(u)=z(T+u)$ for $u\ge0$.
Splitting the integral at $v=u$ gives
\[
 Z=g+k*Z.
\]
Since $m<1$, the standard renewal resolvent representation
\cite{FrancoGyllenbergDiekmann2021,Gripenberg1980,GripenbergLondenStaffans1990,Miller1971,Zhang2010} gives
\[
 Z=g+\mathfrak r_k*g,
 \qquad
 \mathfrak r_k=\sum_{n=1}^\infty k^{*n}.
\]
Because $k\ge0$,
\[
 \int_0^\infty\mathfrak r_k
 =\sum_{n=1}^\infty m^n
 =\frac{m}{1-m}.
\]
Moreover, $g=c+h$, where $h$ is bounded and supported in $[0,a]$.
Thus $h(u)\to0$ and $(\mathfrak r_k*h)(u)\to0$, while
\[
 c\int_0^u\mathfrak r_k(v)\dd v
 \longrightarrow c\frac{m}{1-m}.
\]
This proves \eqref{eq:kernel-resolvent-formula} and
\eqref{eq:kernel-resolvent-limit}.

Finally, consider~\ref{item:kernel-obstruction}. Suppose, to the contrary,
that $c<0$, and set
\[
 M_0:=\max_{T-a\le t\le T}x(t).
\]
We first claim that $x(t)\le M_0$ for every $t\ge T$. Indeed, for any
$R>T$, let
\[
 M_R:=\max_{T-a\le t\le R}x(t).
\]
If $M_R>M_0$, then the maximum is attained at some $t_R\in(T,R]$, and
positivity of $k$ gives
\[
 M_R=x(t_R)\le c+(\mathcal Kx)(t_R)\le c+mM_R<M_R,
\]
a contradiction. Thus $x$ is bounded. Put
\[
 L:=\limsup_{t\to\infty}x(t)\ge0.
\]
For every $\varepsilon>0$, one has $x(t)\le L+\varepsilon$ for all
sufficiently large $t$, and hence
\[
 \limsup_{t\to\infty}(\mathcal Kx)(t)\le mL.
\]
Taking the upper limit in $x\le c+\mathcal Kx$ yields
\[
 L\le c+mL,
\]
so $(1-m)L\le c<0$, impossible. Therefore $c\ge0$.
\end{proof}

Let $\mu$ be a finite nonzero positive Borel measure on $[-h,0]$, let
$a_0\ge0$, and consider
\begin{equation}\label{eq:first-order-general-positive}
 y'(t)+a_0y(t)+\int_{-h}^0y(t+\theta)\dd\mu(\theta)=0.\end{equation}
Its characteristic function~\cite{Hale1977,HaleVerduynLunel1993} is
\begin{equation}\label{eq:first-order-characteristic}
 \chi(\lambda):=\lambda+a_0+\int_{-h}^0e^{\lambda\theta}\dd\mu(\theta).\end{equation}
The following lemma is a generalization of the corresponding
positivity lemma in \cite{PSS2023}.

\begin{lemma}
\label{lem:first-order-projection-positivity}
Let $\rho\in\mathbb R$ be a simple zero of the characteristic function $\chi$
defined in \eqref{eq:first-order-characteristic}, with $\chi'(\rho)>0$.
Let $y$ solve \eqref{eq:first-order-general-positive} with strictly positive
initial history on $[-h,0]$, and set
\begin{equation}\label{eq:first-order-spectral-functional}
 Q_\rho(t):=y(t)-\int_{-h}^0\int_\theta^0
 e^{\rho(\theta-v)}y(t+v)\dd v\dd\mu(\theta),\end{equation}
\begin{equation*} c_1:=\frac{Q_\rho(0)}{\chi'(\rho)}.\end{equation*}
Then
\[
 y(t)>0\quad(t\ge0)\quad\Longleftrightarrow\quad c_1\ge0.
\]
In this case $y'(t)<0$ for every $t\ge0$.
\end{lemma}

\begin{proof}
The equation \eqref{eq:first-order-general-positive} and $\chi(\rho)=0$ give
\begin{equation*} Q_\rho'(t)=\rho Q_\rho(t),\qquad Q_\rho(t)=e^{\rho t}Q_\rho(0).\end{equation*}
Write $Y(t)=e^{-\rho t}y(t)$ and $q_0=Q_\rho(0)$. Applying Fubini's theorem to \eqref{eq:first-order-spectral-functional}, one obtains
\[
 Y(t)=q_0+\int_0^h k_\rho(v)Y(t-v)\dd v,
 \qquad k_\rho(v):=\int_{[-h,-v]}e^{\rho\theta}\dd\mu(\theta).
\]
This nonnegative kernel has mass
\[
 \int_0^h k_\rho(v)\dd v
 =\int_{-h}^0(-\theta)e^{\rho\theta}\dd\mu(\theta)
 =1-\chi'(\rho)<1.
\]
If $\mu([-h,0))>0$, parts~\ref{item:kernel-strict} and~
\ref{item:kernel-obstruction} of Lemma~\ref{lem:positive-kernel-comparison}
give $Y>0$ if and only if $q_0\ge0$. If $\mu([-h,0))=0$, the identity
reduces to $Y(t)=q_0=y(0)>0$. Since $\chi'(\rho)>0$, this proves the
equivalence. Positivity in \eqref{eq:first-order-general-positive} then
gives $y'(t)<0$.
\end{proof}

\begin{theorem}
\label{thm:stable-projection-positivity}
Let $x$ solve \eqref{eq:spectral-autonomous-positive} with strictly positive
initial history on $[-a,0]$, and write
\begin{equation*} c_0:=\frac{J_s[x](0)}{\Delta'(s)},\qquad
 c_1:=\frac{J_{-r}[x](0)}{\Delta'(-r)}.\end{equation*}
Then $c_0=0$ and $c_1\ge0$ if and only if
\[
 x(t)>0,\qquad x'(t)<0,\qquad t\ge0.
\]
\end{theorem}

\begin{proof}
By Remark~\ref{rem:spectral-first-order-reduction}, $c_0=0$ is equivalent to
\eqref{eq:spectral-first-order-reduction}. This equation is of the form
\eqref{eq:first-order-general-positive}, with $a_0=s$ and
$\dd\mu(\theta)=s^2e^{-s\theta}\dd\theta$. Its characteristic function satisfies
\begin{equation*} \chi(\lambda)=\lambda+s+s^2\int_{-a}^0e^{(\lambda-s)\theta}\dd\theta
 =\frac{\Delta(\lambda)}{\lambda-s},\end{equation*}
with the quotient defined at $\lambda=s$ by continuity. In particular,
\begin{equation*} \chi'(-r)=\frac{\Delta'(-r)}{-r-s}
 =\frac{r(2-ar)}{r+s}>0.\end{equation*}
Fubini's theorem in \eqref{eq:first-order-spectral-functional} gives
\begin{equation*} Q_{-r}[x]=\frac{J_s[x]-J_{-r}[x]}{r+s}.\end{equation*}
When $J_s[x](0)=0$, this implies
\[
 \frac{Q_{-r}[x](0)}{\chi'(-r)}
 =\frac{J_{-r}[x](0)}{\Delta'(-r)}=c_1.
\]
Lemma~\ref{lem:first-order-projection-positivity} therefore proves
sufficiency directly. Conversely, positivity and decrease, together with
$x''(t)=x(t-a)>0$, force $x,x'\to0$. Hence $J_s[x](t)\to0$, and
\eqref{eq:J-evolution} gives $c_0=0$. Lemma~\ref{lem:first-order-projection-positivity} now gives $c_1\ge0$.
\end{proof}

\subsection{Comparison estimates for positive decreasing solutions}
\label{subsec:spectral-comparison}

We express the second-order delay operator through a positive integral
operator. This representation yields sign restrictions on the spectral
coordinates of positive decreasing supersolutions.

\begin{lemma}

Let $0<a<2/e$, with $r,s$ as above, and define
\begin{align*}
 h_a(v)&:=\frac{r^2e^{-rv}-s^2e^{sv}}{r+s},\qquad 0\le v\le a,\\
 (\mathcal H_af)(t)&:=\int_0^a h_a(v)f(t-v)\dd v,
 \qquad \mathcal P_a:=(D+r)(D-s),\quad D:=\frac{\mathrm d}{\mathrm dt}.
\end{align*}
The kernel $h_a$ is positive on $[0,a)$, vanishes at $a$, and satisfies
\begin{equation}\label{eq:PSS-memory-masses}
 \int_0^a h_a(v)\dd v=1-\frac1{rs}<1,
 \qquad
 \int_0^a h_a(v)e^{rv}\dd v=1-\frac{r(2-ar)}{r+s}<1.
\end{equation}
For every $C^1$ function $x$ with locally absolutely continuous derivative,
\begin{equation}\label{eq:PSS-history-functional}
 \mathcal C_a[x](t):=J_s[x](t)-J_{-r}[x](t)
 =(r+s)(I-\mathcal H_a)x(t),
\end{equation}
and
\begin{align}
 \mathcal H_a\mathcal P_ax
 &=x(t-a)+(r-s)x'(t)-rsx(t),
 \label{eq:PSS-memory-factor-one}\\
 (I-\mathcal H_a)\mathcal P_ax&=x''(t)-x(t-a).
 \label{eq:PSS-memory-factor-two}
\end{align}
\end{lemma}

\begin{proof}
The positivity of $h_a$ and the identities
\eqref{eq:PSS-memory-masses}--\eqref{eq:PSS-history-functional}
follow directly from the definitions and the characteristic identities
$r^2e^{-ra}=s^2e^{sa}=1$. Moreover,
\[
 h_a(0)=r-s,\qquad h_a(a)=0,\qquad h_a'(a)=-1,
\]
and
\[
 h_a'(0)=-(r^2-rs+s^2),\qquad h_a''+(r-s)h_a'-rsh_a=0.
\]
Two integrations by parts therefore give
\eqref{eq:PSS-memory-factor-one}. Subtracting this identity from
$\mathcal P_ax=x''+(r-s)x'-rsx$ gives
\eqref{eq:PSS-memory-factor-two}.
\end{proof}

\begin{lemma}
\label{lem:PSS-actual-history}
Let $0<a<2/e$. Suppose that $x>0$ is decreasing on $[T-a,\infty)$,
$x,x'\to0$, and $x''(t)\ge x(t-a)$ almost everywhere for $t\ge T$. Then
\begin{equation}\label{eq:PSS-actual-spectral-signs}
 J_s[x](t)\le0,\qquad \mathcal C_a[x](t)\ge0,\qquad t\ge T.\end{equation}
Moreover, if $t_0\in\mathbb R$ and $y$ is any solution of
\[
 y''(t)=y(t-a),\qquad t\ge t_0,
\]
with $C^1$ history on $[t_0-a,t_0]$ and $J_s[y](t_0)=0$, then
\begin{equation}\label{eq:PSS-actual-c1}
 \lim_{t\to\infty}e^{r(t-t_0)}y(t)
 =\frac{\mathcal C_a[y](t_0)}{r(2-ar)}.\end{equation}
This convergence is uniform if the $C^1$ histories on $[t_0-a,t_0]$ range in a bounded subset of $C^1([t_0-a,t_0])$
on which
\[
 A_y:=\frac{\mathcal C_a[y](t_0)}{r(2-ar)}
\]
stays bounded away from zero. Namely, uniformly over these histories,
\[
 \left\|
 \frac{y(t+\cdot)}{A_y e^{-r(t-t_0)}}-e^{-r\cdot}
 \right\|_{C^1([-a,0])}\longrightarrow0.
\]
\end{lemma}

\begin{proof}
Set $f=x''-x(\cdot-a)\ge0$. Integrating \eqref{eq:PSS-J-forced} to infinity,
using $x,x'\to0$, gives
\begin{equation*} J_s[x](t)=-\int_t^\infty e^{-s(u-t)}f(u)\dd u\le0.\end{equation*}
Subtracting \eqref{eq:PSS-J-forced} for $s$ and $-r$ gives
\begin{equation*} \mathcal C_a'[x]+r\mathcal C_a[x]=(r+s)J_s[x]\le0.\end{equation*}
If $\mathcal C_a[x](t_0)<0$, set $W(t)=e^{r(t-t_0)}x(t)$ and
$\delta=-\mathcal C_a[x](t_0)/(r+s)>0$. The preceding inequality gives
\[
 \frac{\mathcal C_a[x](t)}{r+s}
 \le -\delta e^{-r(t-t_0)},\qquad t\ge t_0.
\]
Together with \eqref{eq:PSS-history-functional}, this implies
\[
 0<W(t)\le-\delta+\int_0^a h_a(v)e^{rv}W(t-v)\dd v,
 \qquad t\ge t_0.
\]
Lemma~\ref{lem:positive-kernel-comparison}\textup{\ref{item:kernel-obstruction}},
with the second mass in \eqref{eq:PSS-memory-masses}, rules this out.
Thus \eqref{eq:PSS-actual-spectral-signs} holds.

For the last assertion, \eqref{eq:J-evolution} and $J_s[y](t_0)=0$ give
$J_s[y]\equiv0$, and hence
\[
 \mathcal C_a[y](t)
 =e^{-r(t-t_0)}\mathcal C_a[y](t_0).
\]
Therefore $Y(t):=e^{r(t-t_0)}y(t)$ satisfies
\begin{equation*} Y(t)=\frac{\mathcal C_a[y](t_0)}{r+s}
      +\int_0^a h_a(v)e^{rv}Y(t-v)\dd v.\end{equation*}
Apply Lemma~\ref{lem:positive-kernel-comparison}\textup{\ref{item:kernel-resolvent}} with
\[
 k(v)=h_a(v)e^{rv},\qquad
 c=\frac{\mathcal C_a[y](t_0)}{r+s},\qquad
 m:=\int_0^a h_a(v)e^{rv}\dd v
 =1-\frac{r(2-ar)}{r+s}<1.
\]
The resolvent limit \eqref{eq:kernel-resolvent-limit} gives
\eqref{eq:PSS-actual-c1} directly:
\[
 \lim_{t\to\infty}Y(t)
 =\frac{\mathcal C_a[y](t_0)/(r+s)}{1-m}
 =\frac{\mathcal C_a[y](t_0)}{r(2-ar)}.
\]
Equation~\eqref{eq:kernel-resolvent-formula} gives the stated uniformity. For $C^1$ histories on
$[t_0-a,t_0]$ in a fixed bounded neighborhood, both $c$ and
$h:=g-c$ are uniformly bounded, with $\operatorname{supp}h\subset[0,a]$.
Hence, for $u>a$,
\[
 Y(t_0+u)-A_y
 =\int_{u-a}^{u}\mathfrak r_k(v)h(u-v)\dd v
 -c\int_u^\infty\mathfrak r_k(v)\dd v,
\]
and therefore, with a constant $C$ independent of the history,
\[
 |Y(t_0+u)-A_y|
 \le C\int_{u-a}^\infty\mathfrak r_k(v)\dd v\longrightarrow0.
\]
Thus the convergence is uniform, also on translates $t+[-a,0]$. Since
$J_s[y]\equiv0$, equation~\eqref{eq:spectral-first-order-reduction} expresses
$y'$ continuously in terms of the recent history of $y$, so the derivative
convergence is uniform as well. Finally, the lower bound for $A_y$ allows the
preceding uniform estimates to be divided by $A_y$, yielding the asserted
normalized $C^1$ convergence.\qedhere
\end{proof}

\begin{lemma}[{\cite{Myshkis1972,PSS2023}}]
\label{lem:PSS-rR-properties}
For $0<a<2/e$, consider $r=r(a)$ and $R=R(a)$, the two positive roots of
\begin{equation}
 q^2=e^{aq},
\end{equation}
satisfying
\[
 1<r(a)<e<R(a).
\]
Equivalently, $-r(a)$ and $-R(a)$ are the two real characteristic roots of
$y'(t)+y(t-a/2)=0$, and
\begin{equation}
 r(a)=-\frac2a W_0(-a/2),\qquad
 R(a)=-\frac2a W_{-1}(-a/2).
\end{equation}
As continuous functions of $a$, $r$ is strictly increasing, while $R$ is strictly decreasing. Furthermore,
\[
 \lim_{a\to0^+}r(a)=1,\qquad
 \lim_{a\to0^+}R(a)=+\infty,\qquad
 \lim_{a\to(2/e)^-}r(a)=\lim_{a\to(2/e)^-}R(a)=e.
\]
\end{lemma}

\begin{lemma}
\label{lem:PSS-two-sided-decay}
Let $0<a<2/e$, and let $r=r(a)$ and $R=R(a)$ be as in
Lemma~\ref{lem:PSS-rR-properties}. Suppose that $x>0$ is decreasing on
$[T-a,\infty)$, $x,x'\to0$, and
\[
 x''(t)\ge x(t-a)\qquad\text{for almost every }t\ge T.
\]
Then there exist constants $c,C>0$ such that
\begin{equation}\label{eq:PSS-two-sided-decay}
 c e^{-R(t-T)}\le x(t)\le C e^{-r(t-T)},\qquad t\ge T.
\end{equation}
Moreover,
\begin{equation}\label{eq:PSS-log-derivative-bound}
 r\le \liminf_{t\to\infty}\left(-\frac{x'(t)}{x(t)}\right)\le R.
\end{equation}
\end{lemma}

\begin{proof}
Write $s=s(a)$ and $q=\mathcal C_a[x]/(r+s)$. By
Lemma~\ref{lem:PSS-actual-history} and \eqref{eq:PSS-history-functional},
\[
 x=\mathcal H_a x+q,\qquad q\ge0.
\]
Subtracting \eqref{eq:PSS-J-forced} for $s$ and $-r$ gives
\[
 q'+rq=J_s[x]\le0.
\]
Hence $0\le q(t)\le q(T)e^{-r(t-T)}$ for $t\ge T$.
Set
\[
 m_r:=\int_0^a h_a(v)e^{rv}\dd v<1.
\]
Choose $C>0$ so that
\[
 C\ge\max_{T-a\le t\le T}e^{r(t-T)}x(t),
 \qquad (1-m_r)C\ge q(T).
\]
For $v(t)=Ce^{-r(t-T)}$ we have, for every $t\ge T$,
\[
 (I-\mathcal H_a)(v-x)
 =(1-m_r)Ce^{-r(t-T)}-q(t)\ge0.
\]
The initial history of $v-x$ is nonnegative, so
Lemma~\ref{lem:positive-kernel-comparison}\textup{\ref{item:kernel-ae}}
gives the upper bound in \eqref{eq:PSS-two-sided-decay}.

The mode $z(t)=e^{-R(t-T)}$ solves $z''(t)=z(t-a)$.
Applying \eqref{eq:PSS-memory-factor-two}, and using
$\mathcal P_a z=(R-r)(R+s)z\ne0$, gives
\begin{equation}\label{eq:PSS-fast-kernel-mass}
 \mathcal H_a z=z,
 \qquad \int_0^a h_a(v)e^{Rv}\dd v=1.
\end{equation}
Choose
\[
 c:=\min_{T-a\le t\le T}e^{R(t-T)}x(t)>0.
\]
Then $x-cz\ge0$ on $[T-a,T]$ and
$(I-\mathcal H_a)(x-cz)=q\ge0$ on $[T,\infty)$.
Lemma~\ref{lem:positive-kernel-comparison}\textup{\ref{item:kernel-ae}} gives the lower bound in
\eqref{eq:PSS-two-sided-decay}.

It remains to prove \eqref{eq:PSS-log-derivative-bound}. Set
\[
 \ell:=\liminf_{t\to\infty}\left(-\frac{x'(t)}{x(t)}\right).
\]
The lower bound in \eqref{eq:PSS-two-sided-decay} gives $\ell\le R$: otherwise
$-x'/x\ge R+\delta$ eventually for some $\delta>0$, and integration would give
$x(t)=O(e^{-(R+\delta)t})$, a contradiction. Thus $\ell<\infty$.
For every $\delta>0$, eventually
$-x'/x\ge\ell-\delta$, and hence
\[
 \frac{x(t-a)}{x(t)}
 =\exp\left(\int_{t-a}^t-\frac{x'(u)}{x(u)}\dd u\right)
 \ge e^{a(\ell-\delta)}.
\]
Thus $x''\ge\mu_\delta^2x$ eventually, where
$\mu_\delta=e^{a(\ell-\delta)/2}$. Since
$F:=x'+\mu_\delta x$ satisfies $F'\ge\mu_\delta F$ and $F(t)\to0$, we have
$-x'/x\ge\mu_\delta$ eventually. Therefore
$\ell\ge e^{a(\ell-\delta)/2}$, and letting $\delta\to0^+$ gives
$\ell\ge e^{a\ell/2}$. By Lemma~\ref{lem:PSS-rR-properties},
$\ell\ge r$.

\end{proof}

\begin{remark}
With the same assumptions as in the preceding lemma, the bounds
\[
 \frac{x(t-a)}{x(t)}\le \frac{24}{a^4},\qquad
 -\frac{x'(t)}{x(t)}
 \le \frac{6}{a^3}\left(1-\frac{a^2}{2}\right),
\]
hold eventually. Indeed, by~\eqref{eq:positive-final-integral}, we have
$
 x(t)\ge \int_0^a s\,x(t-a+s)\dd s.
$
Bootstrapping gives
\[
 x(t)\ge \int_0^a\int_0^a
 su\,x(t-2a+s+u)\dd u\dd s
 \ge x(t-a)\int_0^a\int_0^{a-s}su\dd u\dd s.
\]
This proves the first estimate. The second follows from convexity and
$x(t-a+s)\ge x(t)+(a-s)(-x'(t))$.

Assuming nonoscillation, these bounds and banal convergence arguments yield a global positive decreasing supersolution $z$. Corollaries~\ref{cor:rank-one-feedback} and~\ref{cor:positive-supersolution} then imply $\mathcal E(a,b)\le1$. Several improvements to the bounds are possible, but it is not clear to the author what the optimal bounds correspond to.
\end{remark}

\begin{lemma}
\label{lem:PSS-eventual-ODE}
Let $x$ be an eventually positive decreasing solution of $x''(t)=x(\tau(t))$,
and suppose that
\[
 a_*:=\liminf_{t\to\infty}(t-\tau(t))\in(0,2/e].
\]
For every $0<a<a_*$,
\begin{equation}\label{eq:PSS-two-coordinate-ratio}
 x(t-a)\ge r(a)s(a)x(t)
 +\bigl(s(a)-r(a)\bigr)x'(t)
\end{equation}
holds eventually.
\end{lemma}

\begin{proof}
Fix $a<d<a_*$, and take $T$ sufficiently large. Since $x''\ge x$ and
$x,x'\to0$, we have $-x'\ge x$. Hence
\begin{equation}\label{eq:PSS-forcing-lower-bound}
 f(t):=x''(t)-x(t-a)
 \ge x(t-d)-x(t-a)\ge(d-a)x(t).
\end{equation}
Lemma~\ref{lem:PSS-two-sided-decay}, applied with delay $d$, gives
\begin{equation}\label{eq:PSS-solution-lower-bound}
 x(t)\ge c e^{-R(d)(t-T)}.
\end{equation}
Put $w=\mathcal P_ax$ and $z(t)=e^{-R(a)(t-T)}$.
By \eqref{eq:PSS-memory-factor-two} and \eqref{eq:PSS-fast-kernel-mass},
\begin{equation}\label{eq:PSS-w-renewal}
 w=\mathcal H_aw+f,\qquad \mathcal H_az=z.
\end{equation}
Choose $M\ge0$ so that $w+Mz\ge0$ on $[T-a,T]$.
Since
\[
 (I-\mathcal H_a)(w+Mz)=f\ge0,
\]
Lemma~\ref{lem:positive-kernel-comparison}\textup{\ref{item:kernel-ae}}
yields $w\ge-Mz$ on $[T,\infty)$.
Substitution into \eqref{eq:PSS-w-renewal}, together with \eqref{eq:PSS-forcing-lower-bound} and \eqref{eq:PSS-solution-lower-bound}, gives, for $t\ge T+a$,
\[
 w(t)\ge f(t)-Mz(t)
 \ge c(d-a)e^{-R(d)(t-T)}-Me^{-R(a)(t-T)}>0
\]
eventually, because $R(a)>R(d)$ by Lemma~\ref{lem:PSS-rR-properties}. Finally, positivity of $\mathcal H_a$
and \eqref{eq:PSS-memory-factor-one} give, eventually,
\[
 0\le\mathcal H_aw(t)
 =x(t-a)+(r(a)-s(a))x'(t)-r(a)s(a)x(t),
\]
which is \eqref{eq:PSS-two-coordinate-ratio}.
\end{proof}

\subsection{On \texorpdfstring{$\liminf$--$\limsup$}{liminf--limsup} excursions and the critical profile}
\label{subsec:spectral-critical-profile}
We now characterize the thought experiment of Section~\ref{subsec:positive-delay} in terms of incoming and outgoing spectral coefficients.

\smallskip
\noindent\textbf{Definition of excursion.} Fix $0<a<2/e$ and $b>a$. For $(c,y,v)\in\mathbb R^3$, first define $V=V_{c,y,v}$ on $[a,b]$ by
\begin{equation}\label{eq:PSS-terminal-family}
 V''(t)=c,\qquad a<t<b,\qquad V(b)=y,\quad V'(b)=-v.
\end{equation}
Thus
\begin{equation*} V(t)=y+v(b-t)+\frac{c}{2}(b-t)^2,\qquad a\le t\le b.\end{equation*}
Extend $V$ backwards from $a$ using the modes $e^{-rt}$ and $e^{st}$, equivalently by $\mathcal P_aV=0$, with $V,V'$ continuous at $a$. The resulting coefficients remain fixed throughout the backward extension. We call them the \emph{incoming coefficients} $A[V]$ and $B[V]$, so that
\begin{equation*} V(t)=A[V]e^{-rt}+B[V]e^{st},\qquad t\le a.\end{equation*}
Extend $V$ forwards from $b$ by the autonomous equation $V''(t)=V(t-a)$, with $V,V'$ continuous at $b$. Define the \emph{outgoing coefficients}
of $V$ by
\begin{equation*} c_0[V]=\frac{e^{-sb}J_s[V](b)}{\Delta'(s)},\qquad
 c_1[V]=\frac{e^{rb}J_{-r}[V](b)}{\Delta'(-r)}.\end{equation*}
Applying \eqref{eq:PSS-J-transport} to the defect
$c-V(u-a)$ on $[a,b]$ gives directly
\begin{equation}\label{eq:PSS-c1-formula}
\begin{aligned}
 c_0[V]&=B[V]+\frac1{\Delta'(s)}
             \int_a^b e^{-su}\bigl[c-V(u-a)\bigr]\dd u,\\
 c_1[V]&=A[V]+\frac1{\Delta'(-r)}
             \int_a^b e^{ru}\bigl[c-V(u-a)\bigr]\dd u.
\end{aligned}
\end{equation}

\begin{lemma}[{\cite{Hirsch1985,Smith1988}}]
\label{lem:positive-cone-ode}
Let $A=(a_{ij})\in\mathbb R^{2\times2}$ satisfy $a_{12},a_{21}\ge0$,
and let $f\in L^1_{\mathrm{loc}}([t_0,\infty),\mathbb R^2)$ be componentwise
nonnegative. If $z$ solves
\[
 z'(t)=Az(t)+f(t),\qquad z(t_0)\in[0,\infty)^2,
\]
then $z(t)\in[0,\infty)^2$ for every $t\ge t_0$. Any component that is
strictly positive at $t_0$ remains strictly positive for $t\ge t_0$.
If $a_{12},a_{21}>0$ and $z(t_0)\ne0$, then both components are strictly
positive for every $t>t_0$.
\end{lemma}

\begin{lemma}
\label{lem:PSS-backward-excursion}
Suppose $X>0$ is decreasing on $[0,b]$, and
\[
 \mathcal P_aX\ge0\quad\text{on }[0,a],\qquad
 X''\ge X(0)\quad\text{on }[a,b].
\]
Define $U=V_{X(0),X(b),-X'(b)}$. Then
\begin{equation*} 0<U(t)\le X(t)\quad(0\le t\le b),\qquad U(0)\le X(0).\end{equation*}
If $J_s[X](b)\le0$ and $\mathcal C_a[X](b)\ge0$, then
\begin{equation*} J_s[U](b)\le0,\qquad \mathcal C_a[U](b)\ge0.\end{equation*}
\end{lemma}

\begin{proof}
Backward Taylor expansion gives, for $a\le t\le b$,
\[
 X(t)-U(t)=\int_t^b(u-t)\bigl(X''(u)-X(0)\bigr)\dd u\ge0,
 \qquad X'(t)-U'(t)\le0.
\]
On $[0,a]$, use the backward variable $\rho=a-t$. The pair consisting of
$X-U$ and its $\rho$-derivative has nonnegative initial data and satisfies
\[
 \frac{\mathrm d}{\mathrm d\rho}
 \begin{pmatrix}X-U\\(X-U)_\rho\end{pmatrix}
 =\begin{pmatrix}0&1\\rs&r-s\end{pmatrix}
  \begin{pmatrix}X-U\\(X-U)_\rho\end{pmatrix}
  +\begin{pmatrix}0\\(\mathcal P_aX)(a-\rho)\end{pmatrix}.
\]
Lemma~\ref{lem:positive-cone-ode} gives nonnegativity of both components. Moreover,
\begin{align*}
 U(a)&=X(b)-X'(b)(b-a)+\frac{X(0)}{2}(b-a)^2>0,\\
 -U'(a)&=-X'(b)+X(0)(b-a)>0,
\end{align*}
because $X(b)>0$, $X'(b)\le0$, and $X(0)>0$. Applying
Lemma~\ref{lem:positive-cone-ode} to $(U,U_\rho)$ gives $U>0$.
At $b$, the functions have the same value and derivative, and $U\le X$ on
$[b-a,b]$. The positive history kernels in $J_s$ and
\eqref{eq:PSS-history-functional} therefore give
\[
 J_s[U](b)\le J_s[X](b),\qquad
 \mathcal C_a[U](b)\ge\mathcal C_a[X](b).
\]
This completes the proof.
\end{proof}

We now give a spectral description of the evaluation-at-zero functional $\mathcal E$ from Section~\ref{subsec:positive-delay}.

\begin{proposition}
For every $0<a<2/e$ and $b>a$, the family
\eqref{eq:PSS-terminal-family} has a unique member $\Psi_{a,b}$ with
forcing $c=1$ and
\begin{equation*} J_s[\Psi_{a,b}](b)=J_{-r}[\Psi_{a,b}](b)=0.\end{equation*}
It is strictly positive and strictly decreasing on $\mathbb R$.
There is also a zero-forcing member $H$ with
$H(0)=1$, $J_s[H](b)=0$, positive and strictly decreasing on $\mathbb R$,
and with $J_{-r}[H](b)<0$. For some $\alpha_{a,b},\omega_{a,b}>0$, every member
$V=V_{c,y,v}$ satisfies
\begin{equation}\label{eq:PSS-positive-spectral-identity}
 V(0)=c\Psi_{a,b}(0)-\alpha_{a,b}J_s[V](b)
                        +\omega_{a,b}\mathcal C_a[V](b).\end{equation}
\end{proposition}

\begin{proof}
\proofpart{I}{Properties of $P$ and $Q$}
Consider $P=V_{0,1,0}$ and $Q=V_{0,0,1}$. On $[a,b]$,
\[
 P(t)=1,\qquad Q(t)=b-t.
\]
At $t=a$, the corresponding data in the backward variable $\rho=a-t$ are
\[
 (P,P_\rho)=(1,0),\qquad (Q,Q_\rho)=(b-a,1).
\]
For $t\le a$, both functions satisfy $\mathcal P_aV=0$, equivalently
\begin{equation}\label{eq:PQ-backward-system}
 \frac{\mathrm d}{\mathrm d\rho}
 \begin{pmatrix}V\\ V_\rho\end{pmatrix}
 =\begin{pmatrix}0&1\\ rs&r-s\end{pmatrix}
 \begin{pmatrix}V\\ V_\rho\end{pmatrix}.
\end{equation}
Lemma~\ref{lem:positive-cone-ode} gives $P,Q>0$ for $t<b$, with
$P'\le0$ and $Q'<0$. In
particular, $P(0),Q(0)>0$, and the definition of $J_s$ gives $J_s[P](b)>0$.

Let $Q_0$ be the solution of
\[
 \mathcal P_aQ_0=0\quad\text{on }[0,b],
 \qquad Q_0(b)=0,\quad Q_0'(b)=-1.
\]
In the backward variable $\xi=b-t$,
\[
 Q_0(b-\xi)=\frac{e^{r\xi}-e^{-s\xi}}{r+s},\qquad 0\le \xi\le b.
\]
For $0\le \xi\le b-a$, the identity $Q(t)=b-t$ on $[a,b]$ gives
$Q(b-\xi)=\xi$, while $q_0(\xi):=Q_0(b-\xi)$ satisfies
\[
 \frac{\mathrm d}{\mathrm d\xi}
 \begin{pmatrix}q_0\\ q_0'\end{pmatrix}
 =\begin{pmatrix}0&1\\ rs&r-s\end{pmatrix}
  \begin{pmatrix}q_0\\ q_0'\end{pmatrix},
 \qquad (q_0(0),q_0'(0))=(0,1).
\]
Lemma~\ref{lem:positive-cone-ode} gives $q_0,q_0'>0$ for $\xi>0$, so
$q_0''=(r-s)q_0'+rsq_0>0$. Integration from $\xi=0$ yields, for $0\le \xi\le b-a$,
\[
 q_0'(\xi)\ge1=\frac{\mathrm d}{\mathrm d\xi}Q(b-\xi),\qquad
 q_0(\xi)\ge \xi=Q(b-\xi).
\]
If $a>b-a$, then on $[b-a,a]$ the state vectors
\[
 \left(Q(b-\xi),\frac{\mathrm d}{\mathrm d\xi}Q(b-\xi)\right),
 \qquad
 \left(Q_0(b-\xi),\frac{\mathrm d}{\mathrm d\xi}Q_0(b-\xi)\right)
\]
both satisfy \eqref{eq:PQ-backward-system}. Applying
Lemma~\ref{lem:positive-cone-ode} to their difference extends the comparison
from $\xi=b-a$ to $\xi=a$. Thus,
\[
 Q(b-\xi)\le Q_0(b-\xi)=\frac{e^{r\xi}-e^{-s\xi}}{r+s},
 \qquad 0\le \xi\le a.
\]
Substituting this bound into the definition of $J_s[Q](b)$ gives
\[
 J_s[Q](b)=-1+s^2\int_0^a e^{s\xi}Q(b-\xi)\dd \xi
 \le-\frac{\Delta'(s)}{r+s}<0.
\]

\proofpart{II}{The zero-forcing profile $H$}
Define
\begin{equation*} H:=\frac{J_s[P](b)Q-J_s[Q](b)P}{J_s[P](b)Q(0)-J_s[Q](b)P(0)}.\end{equation*}
The denominator is positive, so $H>0$, $H'<0$ for $t\le b$,
$H(0)=1$, and $J_s[H](b)=0$. Its stable incoming coefficient satisfies
\[
 A[H]=\frac{e^{ra}}{r+s}\bigl(sH(a)-H'(a)\bigr)>0.
\]
By \eqref{eq:PSS-c1-formula} with zero forcing, and since $\Delta'(-r)<0$,
\[
 J_{-r}[H](b)
 =e^{-rb}\left(\Delta'(-r)A[H]
 -\int_a^b e^{ru}H(u-a)\dd u\right)<0.
\]
Thus, since $J_s[H](b)=0$,
\begin{equation}\label{eq:PSS-H-positive-coordinate}
 \mathcal C_a[H](b)=-J_{-r}[H](b)>0.\end{equation}
Theorem~\ref{thm:stable-projection-positivity}
gives positivity and strict decrease of its autonomous continuation.

\proofpart{III}{The unit-forcing profile $\Psi_{a,b}$}
To construct the unit-forcing profile, take $S=V_{1,0,0}$. Since $J_s[S](b)>0$ and $J_s[Q](b)<0$, the function
\[
 Z:=S-\frac{J_s[S](b)}{J_s[Q](b)}Q
\]
is positive before $b$, is strictly decreasing there, and satisfies
$Z(b)=0$ and $J_s[Z](b)=0$. Consequently
$\mathcal C_a[Z](b)<0$, by \eqref{eq:PSS-history-functional}. Define
\begin{equation*} \Psi_{a,b}:=Z-\frac{\mathcal C_a[Z](b)}{\mathcal C_a[H](b)}H.\end{equation*}
This positive combination is strictly positive and strictly decreasing for
$t\le b$, has unit forcing, and has $J_s[\Psi_{a,b}](b)=0$ and
$\mathcal C_a[\Psi_{a,b}](b)=0$. Thus both outgoing coefficients vanish.
Theorem~\ref{thm:stable-projection-positivity}, including $c_1=0$, gives
positivity and strict decrease for $t\ge b$.

\proofpart{IV}{Uniqueness and the representation formula}
Since $J_s[H](b)=0$, $J_s[Q](b)<0$, and $\mathcal C_a[H](b)>0$,
the functions $H,Q$ form a basis of the zero-forcing family. If $W=\lambda H+\mu Q$ has both outgoing coefficients zero, then $J_s[W](b)=0$ gives $\mu=0$, and $\mathcal C_a[W](b)=0$ gives $\lambda=0$. This proves uniqueness of $\Psi_{a,b}$.
Finally, since $\mathcal C_a[Q](b)<0$, define
\begin{equation}\label{eq:PSS-positive-transfer-weights}
 \omega_{a,b}:=\frac1{\mathcal C_a[H](b)},\qquad
 \alpha_{a,b}:=\frac{Q(0)-\omega_{a,b}\mathcal C_a[Q](b)}{-J_s[Q](b)}>0.\end{equation}
The identity $W(0)=-\alpha_{a,b}J_s[W](b)+\omega_{a,b}\mathcal C_a[W](b)$
holds on the basis $H,Q$, hence on the zero-forcing family. Apply it to
$W=V-c\Psi_{a,b}$ to obtain \eqref{eq:PSS-positive-spectral-identity}.
\end{proof}

\begin{remark}
The unique member just constructed is precisely $\Phi_{a,b}$ of
Section~\ref{subsec:positive-delay}. Indeed, the incoming representation in
Corollary~\ref{cor:positive-shape} places $\Phi_{a,b}$ in the family \eqref{eq:PSS-terminal-family},
and its decay $\Phi_{a,b},\Phi_{a,b}'=O(e^{-2t/a})$, forces
both outgoing coefficients to vanish by \eqref{eq:J-evolution}. Uniqueness
therefore gives $\Psi_{a,b}=\Phi_{a,b}$ and
$\Psi_{a,b}(0)=\mathcal E(a,b)$.
\end{remark}

\begin{corollary}
\label{cor:PSS-beta-spectral}
Fix $0<a<2/e$, $b>a$, and $c>0$. There is a unique member $U_*$ of
\eqref{eq:PSS-terminal-family} with forcing $c$ such that
\[
 U_*(0)=c,\qquad J_s[U_*](b)=0.
\]
It is given by
\begin{equation}\label{eq:PSS-terminal-representation}
 U_*=c\Phi_{a,b}+c\bigl(1-\mathcal E(a,b)\bigr)H,\end{equation}
and its outgoing coefficients satisfy $c_0[U_*]=0$ and
\begin{equation}\label{eq:PSS-exact-c1-gain}
 c_1[U_*]=c\bigl(1-\mathcal E(a,b)\bigr)c_1[H]
 =\frac{ce^{rb}}{\omega_{a,b}r(2-ar)}\bigl(1-\mathcal E(a,b)\bigr).\end{equation}
For $b\le\beta(a)$ this continuation is positive and strictly decreasing.
For $b>\beta(a)$ it is eventually negative.
\end{corollary}

\begin{proof}
The formula \eqref{eq:PSS-terminal-representation} has the required
forcing, initial value, and vanishing unstable coefficient. If $U_1,U_2$
are two such members, their difference $W:=U_1-U_2$ has zero forcing and
satisfies $J_s[W](b)=0$. Hence $W$ is a multiple of $H$. Since $W(0)=0$,
that multiple vanishes.
Since $\Phi_{a,b}$ has both outgoing coefficients zero,
\eqref{eq:PSS-exact-c1-gain} follows from
\eqref{eq:PSS-H-positive-coordinate} and \eqref{eq:PSS-positive-transfer-weights}.
The crossing property of $b\mapsto\mathcal E(a,b)$ in Proposition~\ref{prop:positive-crossing}, together with \eqref{eq:PSS-exact-c1-gain}, gives the corresponding sign of $c_1[U_*]$. When $\mathcal E(a,b)\le1$, the formula is a nonnegative
linear combination of positive decreasing profiles. When $\mathcal E(a,b)>1$, the limit in Lemma~\ref{lem:PSS-actual-history} gives
$e^{rt}U_*(t)\to c_1[U_*]<0$, proving eventual negativity.
\end{proof}

\begin{remark}
For the normalization $c=1$, both $U_*$ and $\Phi_{a,b}$ solve
\eqref{eq:positive-profile-equation}. Whenever $\mathcal E(a,b)\le1$,
they are furthermore both strictly decreasing and tend to zero. If
$\mathcal E(a,b)<1$, however, they do not have the same decay rate as
$t\to\infty$. This spectral leeway, represented by the
$\bigl(1-\mathcal E(a,b)\bigr)H$ term in
\eqref{eq:PSS-terminal-representation}, will be key in the sharpness
construction below. 
\end{remark}

\begin{lemma}
\label{lem:PSS-half-line-comparison}
Let $\mu,\nu$ be locally integrable on
$[0,\infty)$, with $\nu\ge0$, and let $\sigma:[0,\infty)\to\mathbb R$ be measurable with
$\sigma(t)\le t$. Define
\[
 (Lu)(t)
 :=u''(t)+\mu(t)u'(t)
 -\nu(t)\mathbf1_{\{\sigma(t)\ge0\}}u(\sigma(t)).
\]
Suppose $X(t)>0$ and $X'(t)<0$ for every $t\ge0$, and
\[
 LX\ge0,\qquad
 LX
 \ge LY
 \quad\text{a.e. on }[0,\infty).
\]
If
\[
 X(0)=Y(0),\qquad
 X(t)-Y(t)\longrightarrow0\quad\text{as }t\to\infty,
\]
then $X(t)\le Y(t)$ for every $t\ge0$.
\end{lemma}

\begin{proof}
Set $W=X-Y$. If $W$ is positive somewhere, then $W(0)=0$ and
$W(t)\to0$ imply that it attains a positive maximum $M=W(\omega)$
at some $\omega>0$, where $W'(\omega)=0$. Set $z=W-M$ on
$[0,\omega]$. Since
\[
 L1
 =-\nu(t)\mathbf1_{\{\sigma(t)\ge0\}}\le0,
\]
we have
\[
 Lz
 =LW
  -ML1\ge0,
 \qquad z(0)=-M,
 \qquad z(\omega)=z'(\omega)=0.
\]

We now reduce the equation to standard form. Set
\[
 P(t):=\int_0^t\mu(v)\dd v,
 \qquad
 \phi(t):=\int_0^t e^{-P(v)}\dd v.
\]
The function $\phi$ is strictly increasing. If
$\widehat u(\phi(t))=u(t)$, then
\[
 u''(t)+\mu(t)u'(t)=e^{-2P(t)}\widehat u''(\phi(t))
 \quad\text{a.e.}
\]
Thus, after multiplication by the positive factor $e^{2P(t)}$, the
operator $L$ becomes an operator of the form
\[
 (\widehat L\widehat u)(\xi)
 =\widehat u''(\xi)-R(\xi)\widehat u(\eta(\xi)),
 \qquad R\ge0,
 \qquad 0\le\eta(\xi)\le\xi,
\]
where on the set $\sigma(t)<0$ we take $R(\phi(t))=0$, and on the set
$\sigma(t)\ge0$ we take
\[
 R(\phi(t))=e^{2P(t)}\nu(t),
 \qquad
 \eta(\phi(t))=\phi(\sigma(t)).
\]

We shall apply the classical
comparison theory of such equations
\cite{ABBD,AzbelevZubkoLabovski1973,BDK,Labovski1975,LabovskiThesis}. Choose $T>\omega$ and write
\[
 \Omega:=\phi(\omega),\qquad \Theta:=\phi(T).
\]
The transformed function $\widehat X$ is positive on $[0,\Theta]$,
strictly decreasing, and satisfies
\[
 \widehat L\widehat X\ge0,
 \qquad \widehat X'(\Theta)=e^{P(T)}X'(T)<0.
\]
Thus, the
Wronskian of a fundamental system of $\widehat L u=0$ has no zeros on
$[0,\Theta)$. In particular, it is nonzero on $[0,\Omega]$.

Let $q$ be the solution of the two-point problem
\[
 \widehat Lq=0,\qquad q(0)=1,\qquad q(\Omega)=0.
\]
By the two-point theory \cite{Labovski1975,LabovskiThesis},
$q>0$ on $[0,\Omega)$. Hence $q'(\Omega)\le0$. In fact the strict inequality
$q'(\Omega)<0$ must hold. If $q'(\Omega)=0$, then the Wronskian of $q$ and the
Cauchy function vanishes at $\Omega$, contrary to the preceding
paragraph. Equivalently, the solution would have to be trivial, because of the terminal boundary conditions \cite{ABBD,BDK,LabovskiThesis}.

Set
\[
 f:=\widehat L\widehat z\ge0
 \qquad\text{on }[0,\Omega],
\]
and let $v$ solve
\[
 \widehat Lv=f,\qquad v(0)=v(\Omega)=0.
\]
The same Wronskian condition implies that this two-point problem is
uniquely solvable. Moreover, its Green function $G$ satisfies
\[
 G(\xi,\eta)<0\qquad(0<\xi,\eta<\Omega)
\]
by \cite[Theorem~1]{Labovski1975}. Therefore
\[
 v(\xi)=\int_0^\Omega G(\xi,\eta)f(\eta)\dd\eta\le0,
 \qquad 0\le\xi\le\Omega,
\]
and, since $v(\Omega)=0$, we have $v'(\Omega)\ge0$.

Uniqueness of the two-point problem now gives
\[
 \widehat z=-Mq+v.
\]
Consequently,
\[
 \widehat z'(\Omega)
 =-Mq'(\Omega)+v'(\Omega)>0.
\]
On the other hand,
\[
 \widehat z'(\Omega)=e^{P(\omega)}z'(\omega)=0,
\]
a contradiction. Therefore $X\le Y$ on $[0,\infty)$.
\end{proof}

\begin{lemma}
\label{lem:PSS-three-function-comparison}
Let $0<a<2/e$ and $b>a$. Let $L$ be the operator of
Lemma~\ref{lem:PSS-half-line-comparison} with
\[
 \mu(t)=(r-s)\mathbf1_{(0,a)}(t),
 \qquad
 \nu(t)=rs\mathbf1_{(0,a)}(t)+\mathbf1_{(b,\infty)}(t),
\]
and with $\sigma(t)=t$ on $(0,a)$ and $\sigma(t)=t-a$ on $(b,\infty)$.
Suppose $X>0$, $X'<0$, $X(t)\to0$, and
\[
 LX\ge X(0)\mathbf1_{(a,b)}
 \qquad\text{a.e. on }[0,\infty).
\]
Let
\[
 V:=V_{X(0),X(b),-X'(b)}
\]
be the member of the family \eqref{eq:PSS-terminal-family} determined by the terminal data of $X$,
and let $U_*$ be the member of Corollary~\ref{cor:PSS-beta-spectral} with
forcing $X(0)$. Then
\[
 0<V\le X\quad\text{on }[0,b],
 \qquad V(0)\le X(0),
\]
\[
 J_s[V](b)\le0,
 \qquad \mathcal C_a[V](b)\ge0,
\]
and, on the whole half-line,
\begin{equation*}
 V\le X\le U_*.
\end{equation*}
Moreover,
\begin{equation}\label{eq:PSS-coordinate-comparison}
 0\le \mathcal C_a[X](b)
 \le \mathcal C_a[V](b)
 \le \mathcal C_a[U_*](b).
\end{equation}
\end{lemma}

\begin{proof}
The differential inequality for $X$ says that
$\mathcal P_aX\ge0$ on $(0,a)$, $X''\ge X(0)$ on $(a,b)$, and
$X''(t)\ge X(t-a)$ for $t>b$. Lemma~\ref{lem:PSS-actual-history},
applied from $b$ onward, gives
\[
 J_s[X](b)\le0,\qquad \mathcal C_a[X](b)\ge0.
\]
Lemma~\ref{lem:PSS-backward-excursion} therefore gives
\[
 0<V\le X\quad(0\le t\le b),\qquad V(0)\le X(0),
\]
and
\[
 J_s[V](b)\le0,\qquad \mathcal C_a[V](b)\ge0.
\]
For $t>b$, set $D=X-V$. Then
$D\ge0$ on $[b-a,b]$, $D(b)=D'(b)=0$, and
\[
 D''(t)\ge D(t-a).
\]
Standard representations \cite{AzbelevZubkoLabovski1973,Labovski1975} of solutions to the Cauchy problem at $b$, yield
$D\ge0$ on $[b,\infty)$. Thus $V\le X$ on the whole half-line.

Since $J_s[U_*](b)=0$, Lemma~\ref{lem:PSS-actual-history} gives
$U_*(t)\to0$. Hence Lemma~\ref{lem:PSS-half-line-comparison}, with
$Y=U_*$, yields $X\le U_*$.

We already have $\mathcal C_a[X](b)\ge0$. Since $V(b)=X(b)$ and
$V\le X$ on $[b-a,b]$, the representation
\eqref{eq:PSS-history-functional} gives
\[
 \mathcal C_a[V](b)-\mathcal C_a[X](b)
 =(r+s)\int_0^a h_a(v)\bigl(X(b-v)-V(b-v)\bigr)\dd v\ge0.
\]
Finally, apply \eqref{eq:PSS-positive-spectral-identity} first to $V$ and
then to $U_*$. Using $\Psi_{a,b}(0)=\mathcal E(a,b)$,
$U_*(0)=X(0)$, and $J_s[U_*](b)=0$, we obtain
\[
 \begin{aligned}
 V(0)&=X(0)\mathcal E(a,b)-\alpha_{a,b}J_s[V](b)
       +\omega_{a,b}\mathcal C_a[V](b),\\
 X(0)&=X(0)\mathcal E(a,b)+\omega_{a,b}\mathcal C_a[U_*](b).
 \end{aligned}
\]
Therefore
\[
 \omega_{a,b}\bigl(\mathcal C_a[U_*](b)-\mathcal C_a[V](b)\bigr)
 =X(0)-V(0)-\alpha_{a,b}J_s[V](b)\ge0.
\]
Together with the preceding inequalities this proves
\eqref{eq:PSS-coordinate-comparison}.
\end{proof}

\begin{proof}[\normalfont\normalsize\bfseries Alternative proof of Theorem~\ref{thm:positive-main}]
\leavevmode\par\nobreak\smallskip\noindent
Assume that an eventually positive decreasing solution $x$ exists.
Set
\[
 a_*:=\liminf_{t\to\infty}(t-\tau(t)),\qquad
 b_*:=\limsup_{t\to\infty}(t-\tau(t)).
\]
The cases $a_*=0$ and $a_*>2/e$ are treated by
Propositions~\ref{prop:BDS16} and~\ref{prop:autonomous-liminf}, as in the
first proof. Suppose $0<a_*\le2/e$ and $b_*>\beta(a_*)$. By continuity
of $\beta$, choose
\[
 0<a<a_*,\qquad a<b<b_*,\qquad b>\beta(a).
\]
Choose $t_0$ so that $x>0$ and $x'<0$ on $[t_0,\infty)$.
By the proof of Lemma~\ref{lem:PSS-eventual-ODE}, choose $S\ge t_0$
sufficiently large that
\[
 \tau(u)\ge t_0,\qquad u-\tau(u)\ge a\quad(u\ge S),
 \qquad \mathcal P_ax\ge0\quad\text{a.e. on }[S,\infty).
\]
Since $b<b_*$ and $\tau(T)\to\infty$, choose $T$ with
$T-\tau(T)>b$ and $\xi:=\tau(T)\ge S$. Set
\[
 X(t)=\frac{x(\xi+t)}{x(\xi)},\qquad t\ge0.
\]
Then $X(0)=1$, $X>0$, $X'<0$, and $X(t)\to0$.
On $(0,a)$, we have $\mathcal P_aX\ge0$.
For $a<t<b$, monotonicity gives
$\tau(\xi+t)\le\tau(T)=\xi$, hence $X''\ge1$.
For $t>b$, we have
$\tau(\xi+t)\le\xi+t-a$, hence $X''(t)\ge X(t-a)$.
Thus
\[
 LX\ge \mathbf1_{(a,b)}
 \qquad\text{a.e. on }[0,\infty).
\]
Let $V:=V_{1,X(b),-X'(b)}$ and let $U_*$ be the member in
Corollary~\ref{cor:PSS-beta-spectral} with $c=1$.
Lemma~\ref{lem:PSS-three-function-comparison} gives
\[
 0<V\le X\quad\text{on }[0,b],\qquad V(0)\le1,
\]
and
\[
 J_s[V](b)\le0,\qquad \mathcal C_a[V](b)\ge0.
\]
There are now three ways to finish the argument.\medskip

\noindent(1) \eqref{eq:PSS-positive-spectral-identity} applied to $V$ gives
\[
 1\ge V(0)
 =\mathcal E(a,b)-\alpha_{a,b}J_s[V](b)
  +\omega_{a,b}\mathcal C_a[V](b)
 \ge \mathcal E(a,b).
\]
Thus $\mathcal E(a,b)\le1$.

\medskip
\noindent(2) \eqref{eq:PSS-coordinate-comparison} gives
\[
 0\le \mathcal C_a[V](b)\le\mathcal C_a[U_*](b).
\]
Since $J_s[U_*](b)=0$, we have
$J_{-r}[U_*](b)=-\mathcal C_a[U_*](b)\le0$. As $\Delta'(-r)<0$,
$c_1[U_*]\ge0$, and \eqref{eq:PSS-exact-c1-gain} yields
\[
 0\le c_1[U_*]=\bigl(1-\mathcal E(a,b)\bigr)c_1[H],\qquad c_1[H]>0.
\]
Thus again $\mathcal E(a,b)\le1$.

\medskip
\noindent(3) Lemma~\ref{lem:PSS-three-function-comparison}
gives $X\le U_*$ on the whole half-line. Since $b>\beta(a)$,
Corollary~\ref{cor:PSS-beta-spectral} gives $c_1[U_*]<0$, so $U_*$ is
eventually negative. This contradicts $X>0$.

Hence in every case we obtain a contradiction to $b>\beta(a)$, and therefore
$b_*\le\beta(a_*)$.
\end{proof}

\subsection{Sharpness}
\label{subsec:spectral-excursions-sharpness}

\begin{proof}[\normalfont\normalsize\bfseries Alternative proof of Theorem~\ref{thm:positive-sharpness}]
\leavevmode\par\nobreak\smallskip\noindent
By Proposition~\ref{prop:BDS17}, it is enough to prove the endpoint case $b=\beta(a)$. 
Parts~I--III treat the case  $0<a<2/e$. The case $a=0$ is handled in Part~IV. The case $a=2/e$ is trivial.

\proofpart{I}{Overview of the construction}
Choose
\[
 a<b_1<b_2<\cdots<\beta(a),\qquad b_j\longrightarrow\beta(a).
\]
For each $j$, let $U_{b_j}$ denote the member $U_*$ of
Corollary~\ref{cor:PSS-beta-spectral} with $c=1$ and terminal point $b_j$.
Write its incoming part as
\[
 U_{b_j}(t)=A_{b_j}e^{-rt}+B_{b_j}e^{st},\qquad t\le a,
\]
and set
\[
 \rho_j:=\frac{B_{b_j}}{A_{b_j}}.
\]
Part~II shows that $\rho_j\in(-1,0)$. Then
\[
 U_{b_j}(\theta)
 =\frac{e^{-r\theta}+\rho_j e^{s\theta}}{1+\rho_j},
 \qquad -a\le\theta\le0.
\]
At the start of recovery, let $D>0$ and $G$ denote the coefficients of
$e^{-rt}$ and $e^{st}$, and write $q=G/D$. At time $t$, the ratio of their
contributions is
\[
 \rho(t):=\frac{Ge^{st}}{De^{-rt}}=q e^{(r+s)t}.
\]
The canonical excursion has $q=0$. We perturb its
endpoint from $b_j$ to $\ell$ to obtain $q=-\varepsilon<0$. Waiting for a time $\Lambda$ such that
\[
 \rho(\Lambda)=-\varepsilon e^{(r+s)\Lambda}=\rho_{j+1}
\]
makes the ratio of the two modal contributions equal to that of the
next canonical incoming profile. As $\varepsilon\to0^+$, the required recovery
time tends to infinity. Meanwhile, the projection onto the invariant hyperplane $\Gamma_s$ of Remark~\ref{rem:spectral-first-order-reduction}, after normalization, converges uniformly to $e^{-rt}$.

Let
\[
 X_1:=\{\varphi\in C^1([-a,0]):\varphi(0)=1\},
\]
and equip $X_1$ with the topology inherited from $C^1([-a,0])$. For
$\varphi,\psi\in X_1$ write
\[
 \|\varphi-\psi\|_{X_1}:=\|\varphi-\psi\|_{C^1([-a,0])}.
\]
For each $j$, fix a sufficiently small neighborhood $\mathcal V_j$ of
$U_{b_j}$ in $X_1$. Parts~II and~III show that these neighborhoods can be
chosen to have the following return property. If
$\mathcal R_{j,\varepsilon}(u)$ denotes the normalized history obtained
from $u\in\mathcal V_j$ after the corresponding excursion and recovery,
then
\[
 \sup_{u\in\mathcal V_j}
 \|\mathcal R_{j,\varepsilon}(u)-U_{b_{j+1}}\|_{X_1}
 \longrightarrow0
 \qquad(\varepsilon\to0^+).
\]
Thus this one fixed source neighborhood can be carried into \emph{any prescribed
neighborhood, however small}, of $U_{b_{j+1}}$ by taking $\varepsilon$
sufficiently small. For suitable $\bar\varepsilon_j>0$, chosen in Part~III,
we may also arrange that
$
 |\ell(u,\varepsilon)-b_j|
 \longrightarrow0$ (uniformly in $u\in \mathcal V_{j},\varepsilon\le \bar\varepsilon_j$).

Take $T_1=0$ and prescribe the initial history
$x(\theta)=U_{b_1}(\theta)$ for $-a\le\theta\le0$.
Inductively, suppose that the normalized history at $T_j$, denoted by
$u_j$, belongs to $\mathcal V_j$. Choose
$0<\varepsilon_j\le\bar\varepsilon_j$ so small that the normalized history at the end
of the recovery belongs to $\mathcal V_{j+1}$. Put
\[
 \ell_j:=\ell(u_j,\varepsilon_j),
 \qquad
 T_{j+1}=T_j+\ell_j+\Lambda_j,
\]
where $\Lambda_j$ is the corresponding recovery length. Continuing inductively
constructs a strictly positive decreasing solution on $[0,\infty)$, and
$\ell_j\to\beta(a)$.

Define
\[
 \tau(t)=
 \begin{cases}
  T_j,&T_j+a\le t\le T_j+\ell_j,\\
  t-a,&\text{otherwise}
 \end{cases}.
\]
The function $\tau$ is nondecreasing and satisfies $\tau(t)\le t$, and
\[
 \liminf_{t\to\infty}(t-\tau(t))=a,
 \qquad
 \limsup_{t\to\infty}(t-\tau(t))=\beta(a).
\]

\proofpart{II}{The excursion}

\smallskip
\noindent\emph{The excursion for $U_{b_j}$.}
For each $j$, write the incoming part of the profile $U_{b_j}$ introduced in
Part~I as
\[
 U_{b_j}(t)=A_{b_j}e^{-rt}+B_{b_j}e^{st},\qquad t\le a.
\]
Since $U_{b_j}(0)=1$, one has $A_{b_j}+B_{b_j}=1$. Moreover,
\eqref{eq:PSS-c1-formula} and $c_0[U_{b_j}]=0$ give
\[
 B_{b_j}=-\frac1{\Delta'(s)}\int_a^{b_j}
 e^{-sv}\bigl[1-U_{b_j}(v-a)\bigr]\dd v<0.
\]
Hence $A_{b_j}>0$, so
\[
 \rho_j=\frac{B_{b_j}}{A_{b_j}}\in(-1,0),
\]
as asserted in Part~I. At the end of the excursion,
\[
 \frac{J_{-r}[U_{b_j}](b_j)}{\Delta'(-r)}
 =e^{-rb_j}c_1[U_{b_j}]>0,
\]
where positivity follows from $b_j<\beta(a)$.

\smallskip
\noindent\emph{Perturbations of the incoming history and excursion endpoint.}
Let $u\in X_1$ be sufficiently close to $U_{b_j}$, and let $\ell>a$ be a plateau endpoint near $b_j$. Starting from $u$, solve
\[
 y''(t)=y(t-a),\quad 0<t<a,
 \qquad
 y''(t)=1,\quad a<t<\ell,
\]
and continue from $\ell$ by $y''(t)=y(t-a)$. Put
\[
 D(u,\ell):=
 \frac{J_{-r}[y](\ell)}{\Delta'(-r)},\qquad
 G(u,\ell):=
 \frac{J_s[y](\ell)}{\Delta'(s)},\qquad
 q(u,\ell):=\frac{G(u,\ell)}{D(u,\ell)}.
\]
At $(U_{b_j},b_j)$ one has $D(U_{b_j},b_j)=e^{-rb_j}c_1[U_{b_j}]>0$ and $q=0$. Furthermore,
using \eqref{eq:PSS-J-forced} on the plateau,
\[
 \frac{\partial G}{\partial\ell}(u,\ell)
 =\frac{sJ_s[y](\ell)+1-y(\ell-a)}{\Delta'(s)}.
\]
Hence
\[
 \left.\frac{\partial G}{\partial\ell}\right|_{(U_{b_j},b_j)}
 =\frac{1-U_{b_j}(b_j-a)}{\Delta'(s)}>0.
\]
Here $b_j-a>0$ and $U_{b_j}$ is strictly decreasing with $U_{b_j}(0)=1$.
For $\ell=b_j$, $G(U_{b_j},b_j)=0$, so
\[
 \left.\frac{\partial q}{\partial\ell}\right|_{(U_{b_j},b_j)}
 =\frac{1-U_{b_j}(b_j-a)}{D(U_{b_j},b_j)\Delta'(s)}>0.
\]
Near $(U_{b_j},b_j)$, the maps $D$ and $G$ are jointly $C^1$ in the
Fr\'echet sense. Since $D>0$ there, the same holds for $q=G/D$. Thus the
implicit-function theorem, applied to
$q(u,\ell)+\varepsilon=0$, gives a unique $C^1$ map
\[
 (u,\varepsilon)\longmapsto\ell(u,\varepsilon),
 \qquad
 \ell(U_{b_j},0)=b_j,
\]
defined for $u$ sufficiently close to $U_{b_j}$ and
$\varepsilon$ sufficiently small, such that
\[
 D(u,\ell(u,\varepsilon))>0,
 \qquad
 q(u,\ell(u,\varepsilon))=-\varepsilon.
\]
Since $b_j<\beta(a)$, Corollary~\ref{cor:PSS-beta-spectral} gives
$U_{b_j}>0$ and $U_{b_j}'<0$. By continuous dependence in $X_1\times\mathbb R$
on $(u,\ell)$, after shrinking $\mathcal V_j$ and the interval of
admissible endpoints $\ell$ around $b_j$, the corresponding excursion is
also strictly positive and strictly decreasing.

\proofpart{III}{The recovery}

\smallskip
\noindent\emph{Application of Lemma~\ref{lem:PSS-actual-history}.}
Fix $j$. After shrinking $\mathcal V_j$ if necessary, the application of the implicit function theorem
above gives, for every $u\in\mathcal V_j$ and
$0\le\varepsilon\le\varepsilon_j^0$, an endpoint
$\ell=\ell(u,\varepsilon)$ near $b_j$. Shift the origin to the end of the excursion and write
$y_{u,\varepsilon}$ for the recovery solution and
\[
 D_{u,\varepsilon}:=D\bigl(u,\ell(u,\varepsilon)\bigr)>0,
 \qquad
 G_{u,\varepsilon}:=G\bigl(u,\ell(u,\varepsilon)\bigr)
 =-\varepsilon D_{u,\varepsilon}.
\]
Set
\[
 z_{u,\varepsilon}(t)
 :=y_{u,\varepsilon}(t)-G_{u,\varepsilon}e^{st}.
\]
The initial history of $z_{u,\varepsilon}$ belongs to the invariant hyperplane $\Gamma_s$ of Remark~\ref{rem:spectral-first-order-reduction}, so $J_s[z_{u,\varepsilon}](0)=0$. Since
$z_{u,\varepsilon}$ satisfies the autonomous delay equation during the
recovery, \eqref{eq:J-evolution} gives
$J_s[z_{u,\varepsilon}](t)=0$ for all $t\ge0$. Since
$J_{-r}[e^{st}]=0$, we have
\[
 \frac{\mathcal C_a[z_{u,\varepsilon}](0)}{r(2-ar)}
 =-\frac{J_{-r}[y_{u,\varepsilon}](0)}{r(2-ar)}
 =D_{u,\varepsilon}.
\]
By shrinking $\mathcal V_j$ and
$\varepsilon_j^0$ once more, all recovery-start histories of
$z_{u,\varepsilon}$ obtained in this way lie in a fixed sufficiently
small bounded $C^1$-neighborhood and $D_{u,\varepsilon}$ stays bounded
away from zero. Hence the uniformity
assertion in Lemma~\ref{lem:PSS-actual-history} gives
\begin{equation}\label{eq:recovery-uniform-stable-mode}
 \sup_{\substack{u\in\mathcal V_j\\0\le\varepsilon\le\varepsilon_j^0}}
 \left\|
 \frac{z_{u,\varepsilon}(t+\cdot)}
 {D_{u,\varepsilon}e^{-rt}}-e^{-r\cdot}
 \right\|_{C^1([-a,0])}
 \longrightarrow0
 \qquad(t\to\infty).
\end{equation}

\smallskip
\noindent\emph{Spectral convergence during recovery.}
Set
\[
 \Lambda_j(\varepsilon)
 :=\frac1{r+s}\log\frac{-\rho_{j+1}}{\varepsilon}.
\]
For sufficiently small $\varepsilon$, one has $\Lambda_j(\varepsilon)>0$ and
$\Lambda_j(\varepsilon)\to\infty$ as $\varepsilon\to0^+$. For $0\le t\le \Lambda_j(\varepsilon)$ define
\[
 \rho(t):=\frac{G_{u,\varepsilon}e^{st}}
 {D_{u,\varepsilon}e^{-rt}}
 =-\varepsilon e^{(r+s)t}.
\]
Then
\[
 \rho_{j+1}\le\rho(t)<0,
 \qquad
 \rho(\Lambda_j(\varepsilon))=\rho_{j+1}.
\]
Hence
\[
 \frac{y_{u,\varepsilon}(t+\theta)}
 {D_{u,\varepsilon}e^{-rt}}
 =\frac{z_{u,\varepsilon}(t+\theta)}
 {D_{u,\varepsilon}e^{-rt}}+\rho(t)e^{s\theta},
 \qquad -a\le\theta\le0.
\]
Since $\Lambda_j(\varepsilon)\to\infty$ as $\varepsilon\to0^{+}$,
\eqref{eq:recovery-uniform-stable-mode} gives
\[
 \sup_{u\in\mathcal V_j}
 \left\|
 \frac{z_{u,\varepsilon}(\Lambda_j(\varepsilon)+\cdot)}
 {D_{u,\varepsilon}e^{-r\Lambda_j(\varepsilon)}}-e^{-r\cdot}
 \right\|_{C^1([-a,0])}
 \longrightarrow0
 \qquad(\varepsilon\to0^{+}).
\]
Since
\[
 \frac{G_{u,\varepsilon}e^{s\Lambda_j(\varepsilon)}}
 {D_{u,\varepsilon}e^{-r\Lambda_j(\varepsilon)}}
 =\rho_{j+1},
\]
it follows uniformly for $u\in\mathcal V_j$ that
\[
 \frac{y_{u,\varepsilon}(\Lambda_j(\varepsilon)+\theta)}
 {D_{u,\varepsilon}e^{-r\Lambda_j(\varepsilon)}}
 =e^{-r\theta}+\rho_{j+1}e^{s\theta}+o_{C^1}(1),
 \qquad -a\le\theta\le0.
\]
In particular,
\[
 \frac{y_{u,\varepsilon}(\Lambda_j(\varepsilon))}
 {D_{u,\varepsilon}e^{-r\Lambda_j(\varepsilon)}}
 =1+\rho_{j+1}+o(1)
\]
uniformly for $u\in\mathcal V_j$. Since
$1+\rho_{j+1}>0$, division by the last expression gives
\begin{equation}\label{eq:recovery-return-X1}
 \sup_{u\in\mathcal V_j}
 \left\|
 \frac{y_{u,\varepsilon}(\Lambda_j(\varepsilon)+\cdot)}
 {y_{u,\varepsilon}(\Lambda_j(\varepsilon))}
 -U_{b_{j+1}}
 \right\|_{X_1}
 \longrightarrow0
 \qquad(\varepsilon\to0^{+}).
\end{equation}

\smallskip
\noindent\emph{Sign and monotonicity during recovery.}
For $\rho\in[\rho_{j+1},0]$ and $-a\le\theta\le0$, set
\[
 h_\rho(\theta):=e^{-r\theta}+\rho e^{s\theta}.
\]
Then
\begin{equation}\label{eq:recovery-two-mode-bounds}
 h_\rho(\theta)\ge1+\rho_{j+1}>0,
 \qquad
 h_\rho'(\theta)
 =-re^{-r\theta}+s\rho e^{s\theta}\le-r<0.
\end{equation}
Set
\[
 \eta_j:=\frac12\min\{1+\rho_{j+1},r\}>0.
\]
By \eqref{eq:recovery-uniform-stable-mode}, there is a number $R_j>0$, independent of
$u\in\mathcal V_j$ and of
$0\le\varepsilon\le\varepsilon_j^0$, such that
\[
 \sup_{\substack{u\in\mathcal V_j\\0\le\varepsilon\le\varepsilon_j^0}}
 \left\|
 \frac{z_{u,\varepsilon}(t+\cdot)}
 {D_{u,\varepsilon}e^{-rt}}-e^{-r\cdot}
 \right\|_{C^1([-a,0])}
 <\eta_j,
 \qquad t\ge R_j.
\]
On the fixed interval $[0,R_j]$, the canonical recovery corresponding to
$(U_{b_j},0)$ is strictly positive and strictly decreasing. Continuous
dependence allows $\mathcal V_j$ and $\varepsilon_j^0$ to be shrunk so that
the same signs hold for all $u\in\mathcal V_j$ and
$0\le\varepsilon\le\varepsilon_j^0$. 
Now choose $\bar\varepsilon_j\in(0,\varepsilon_j^0]$ so small that
\[
 \Lambda_j(\varepsilon)>R_j
 \qquad(0<\varepsilon\le\bar\varepsilon_j).
\]
Fix any $u\in\mathcal V_j$ and any
$0<\varepsilon\le\bar\varepsilon_j$. For
$R_j\le t\le \Lambda_j(\varepsilon)$ and $-a\le\theta\le0$, put
\[
 E_{u,\varepsilon,t}(\theta)
 :=\frac{z_{u,\varepsilon}(t+\theta)}
 {D_{u,\varepsilon}e^{-rt}}-e^{-r\theta}.
\]
Then $\|E_{u,\varepsilon,t}\|_{C^1([-a,0])}<\eta_j$, while
\[
 \frac{y_{u,\varepsilon}(t+\theta)}
 {D_{u,\varepsilon}e^{-rt}}
 =h_{\rho(t)}(\theta)+E_{u,\varepsilon,t}(\theta).
\]
Differentiating this identity with respect to $\theta$ and using
\eqref{eq:recovery-two-mode-bounds}, together with
$\|E_{u,\varepsilon,t}\|_{C^1([-a,0])}<\eta_j
\le\frac12\min\{1+\rho_{j+1},r\}$, gives
\[
 \frac{y_{u,\varepsilon}(t+\theta)}
 {D_{u,\varepsilon}e^{-rt}}
 >\frac{1+\rho_{j+1}}2>0,
 \qquad
 \frac{y_{u,\varepsilon}'(t+\theta)}
 {D_{u,\varepsilon}e^{-rt}}
 <-\frac r2<0.
\]
These inequalities hold for $R_j\le t\le\Lambda_j(\varepsilon)$ and
$-a\le\theta\le0$.
Thus, for every $0<\varepsilon\le\bar\varepsilon_j$, the entire recovery
is strictly positive and strictly decreasing, uniformly for
$u\in\mathcal V_j$. Together with \eqref{eq:recovery-return-X1}, this
proves the return property stated in Part~I.

\proofpart{IV}{The case $a=0$}
Here $\beta(0)=\sqrt2$, and the return is exact because off the plateaus the
equation is the ordinary differential equation $y''=y$. Choose
\[
 0<b_1<b_2<\cdots<\beta(0),\qquad b_j\longrightarrow\beta(0).
\]
\smallskip
\noindent\emph{The excursion for $U_{b_j}$.}
For each $j$, put
\[
 v_{b_j}:=-\frac{b_j^2+2b_j+2}{2(b_j+1)}
\]
and define the canonical excursion by
\[
 U_{b_j}(t)=1+v_{b_j}t+\frac{t^2}{2},\qquad 0\le t\le b_j.
\]
Extend it to $t\le0$ by the equation $y''=y$. Thus
\[
 U_{b_j}(t)=A_{b_j}e^{-t}+B_{b_j}e^t,\qquad
 A_{b_j}=\frac{1-v_{b_j}}2,\qquad B_{b_j}=\frac{1+v_{b_j}}2,
\]
and
\[
 \rho_j:=\frac{B_{b_j}}{A_{b_j}}
 =-\frac{b_j^2}{(b_j+2)^2}\in(-1,0).
\]
At the end of the excursion,
\[
 U_{b_j}(b_j)=\frac{2-b_j^2}{2(b_j+1)}>0,
 \qquad
 U_{b_j}'(b_j)=-U_{b_j}(b_j).
\]
Since $U_{b_j}'$ is increasing and is still negative at $b_j$, the whole
excursion is strictly decreasing, and hence
$U_{b_j}\ge U_{b_j}(b_j)>0$ on $[0,b_j]$.

\smallskip
\noindent\emph{The excursion for perturbed lengths.}
Fix $j$ and stop the $b_j$-excursion at a point $\ell$ near $b_j$. Put
\[
 Y_j(\ell):=1+v_{b_j}\ell+\frac{\ell^2}{2},\qquad
 P_j(\ell):=v_{b_j}+\ell.
\]
After shifting the origin to $\ell$, the recovery has the form
\[
 y(t)=D_j(\ell)e^{-t}+G_j(\ell)e^t,
\]
where
\[
 D_j(\ell)=\frac{Y_j(\ell)-P_j(\ell)}2,
 \qquad
 G_j(\ell)=\frac{Y_j(\ell)+P_j(\ell)}2.
\]
At $\ell=b_j$ one has $D_j(b_j)=U_{b_j}(b_j)>0$ and $G_j(b_j)=0$. Moreover,
\[
 \frac{\mathrm d}{\mathrm d\ell}G_j(\ell)\Big|_{\ell=b_j}=\frac{1-U_{b_j}(b_j)}2>0.
\]
Consequently, for
\[
 q_j(\ell):=\frac{G_j(\ell)}{D_j(\ell)},
\]
one has
\[
 q_j(b_j)=0,
 \qquad
 \frac{\mathrm d}{\mathrm d\ell}q_j(\ell)\Big|_{\ell=b_j}=\frac{1-U_{b_j}(b_j)}{2U_{b_j}(b_j)}>0.
\]
Thus, for every sufficiently small $\varepsilon>0$, there is a unique
$\ell_j(\varepsilon)<b_j$ close to $b_j$ such that
\[
 q_j(\ell_j(\varepsilon))=-\varepsilon,
 \qquad
 \ell_j(\varepsilon)\longrightarrow b_j
 \quad(\varepsilon\to0^{+}).
\]
The shortened excursion remains strictly positive and strictly decreasing.

\smallskip
\noindent\emph{The recovery and exact return.}
Write $D=D_j(\ell_j(\varepsilon))>0$ and $G=-\varepsilon D$. During the
recovery,
\[
 y(t)=De^{-t}+Ge^t
     =De^{-t}\bigl(1+\rho(t)\bigr),
 \qquad
 \rho(t):=\frac{Ge^t}{De^{-t}}=-\varepsilon e^{2t}.
\]
Choose
\[
 \Lambda_j(\varepsilon):=\frac12\log\frac{-\rho_{j+1}}{\varepsilon}.
\]
For sufficiently small $\varepsilon>0$, one has $\Lambda_j(\varepsilon)>0$,
and $\Lambda_j(\varepsilon)\to\infty$ as $\varepsilon\to0^+$. Moreover,
\[
 \rho_{j+1}\le\rho(t)<0,
 \qquad 0\le t\le \Lambda_j(\varepsilon),
 \qquad
 \rho(\Lambda_j(\varepsilon))=\rho_{j+1}.
\]
Hence throughout the recovery
\[
 y(t)=De^{-t}(1+\rho(t))
 \ge De^{-t}(1+\rho_{j+1})>0,
\]
while
\[
 y'(t)=De^{-t}(-1+\rho(t))<0.
\]
At $t=\Lambda_j(\varepsilon)$, the logarithmic derivative is exactly
\[
 \frac{y'(\Lambda_j(\varepsilon))}{y(\Lambda_j(\varepsilon))}
 =\frac{-1+\rho_{j+1}}{1+\rho_{j+1}}
 =v_{b_{j+1}}.
\]
This is the incoming logarithmic derivative of the next canonical excursion.

\smallskip
\noindent\emph{Concatenation of the excursions and recoveries.}
The rest of the construction is the same as in the case $a>0$. Choose
$\varepsilon_j>0$ sufficiently small so that
$\ell_j(\varepsilon_j)-b_j\to0$ as $j\to\infty$.
Concatenating the excursions and recoveries gives a strictly positive,
strictly decreasing solution with a nondecreasing delayed argument. Since
$\ell_j(\varepsilon_j)\to\beta(0)$ and the delay vanishes during each recovery,
\[
 \liminf_{t\to\infty}(t-\tau(t))=0,
 \qquad
 \limsup_{t\to\infty}(t-\tau(t))=\beta(0).
\]

\end{proof}

\begin{acknowledgements}
The AI tools Claude (Anthropic) and ChatGPT (OpenAI) were used during the prepa-
ration of this manuscript. The author has reviewed all AI-assisted content and takes full
responsibility for the final manuscript.
\end{acknowledgements}

\end{document}